\documentclass[11pt,english]{article}

\usepackage{lmodern}
\usepackage[T1]{fontenc}
\usepackage[latin9]{inputenc}
\usepackage{geometry}
\usepackage{amsfonts}
\usepackage{soul}
\usepackage[toc,page]{appendix}
\usepackage{graphicx}
\usepackage{graphics}
\usepackage{makecell}
\usepackage{dsfont}
\usepackage{mathrsfs}
\usepackage{amsmath}
\usepackage{amssymb}
\usepackage{amsthm}
\usepackage{physics}
\usepackage{bm}
\usepackage{latexsym}
\usepackage{rotating}
\usepackage{sidecap}
\usepackage{enumitem}
\usepackage[labelfont=bf,tableposition=top,figureposition=bottom]{caption}
\DeclareCaptionLabelSeparator*{spaced}{\\[2ex]}
\usepackage{adjustbox}
\usepackage{lscape}
\usepackage{float}
\usepackage{multirow}
\usepackage{hanging}
\usepackage{amsmath}
\usepackage{array}
\usepackage{amssymb}
\usepackage{tabularx}
\usepackage{setspace}
\usepackage{pgfplots}
\usepackage{tikz}
\usepackage{color}
\usepackage{subcaption}
\usepackage{natbib}
\usepackage{authblk}
\usepackage{color}
\usepackage[lined,boxed,linesnumbered,commentsnumbered,ruled,vlined]{algorithm2e}
\SetAlFnt{\smaller}
\SetArgSty{textnormal}
\usepackage{hyperref}
\usepackage{url}
\usepackage{pdfpages}
\usepackage{xcolor, soul}
\definecolor{antiquewhite}{rgb}{0.98, 0.72, 0.54}
\sethlcolor{antiquewhite}

\pgfplotsset{compat=1.10}
\usepgfplotslibrary{groupplots}
\newtheorem{theorem}{Theorem}

\newtheorem{corollary}[theorem]{Corollary}

\newtheorem{proposition}{Proposition}
\newtheorem{remark}{Remark}

\newtheorem{ResearchQuestion}{Q}

\newlength{\Oldarrayrulewidth}
\newcommand{\Cline}[2]{
  \noalign{\global\setlength{\Oldarrayrulewidth}{\arrayrulewidth}}
  \noalign{\global\setlength{\arrayrulewidth}{#1}}\cline{#2}
  \noalign{\global\setlength{\arrayrulewidth}{\Oldarrayrulewidth}}}

\def\maximize{\mathop{\rm maximize}}

\def\argmax{\mathop{\rm arg\,max}}

\usepackage{algpseudocode}

\algnewcommand{\algorithmicand}{\textbf{ and }}
\algnewcommand{\algorithmicor}{\textbf{ or }}
\algnewcommand{\OR}{\algorithmicor}
\algnewcommand{\AND}{\algorithmicand}
\algnewcommand{\var}{\texttt}

\usepackage{titlesec}

\titlespacing*{\section} {0pt}{3ex plus 1ex minus .2ex}{1.5ex plus .2ex}
\titlespacing*{\subsection} {0pt}{2.5ex plus 1ex minus .2ex}{1.25ex plus .2ex}
\titlespacing*{\subsubsection}{0pt}{2.25ex plus 1ex minus .2ex}{1ex plus .2ex}
\titlespacing*{\paragraph} {0pt}{2.5ex plus 1ex minus .2ex}{1em}

\usepackage{setspace}
\providecommand{\keywords}[1]
{
  \small	
  \textbf{\textit{Keywords:}} #1
}

\begin{document}

\clearpage
\title{\vspace{-1.5cm}Increasing Supply Chain Resiliency Through Equilibrium Pricing and Stipulating Transportation Quota Regulation}

\date{\vspace{0.1cm}}
\author{Mostafa Pazoki%
\thanks{Edwards School of Business, University of Saskatchewan, Saskatoon, SK, Canada. Email: \href{mostafa.pazoki@gmail.com}{mostafa.pazoki@gmail.com}}  \:, Hamed Samarghandi%
\thanks{\emph{Corresponding Author}; Edwards School of Business, University of Saskatchewan, Saskatoon, SK, Canada. Email: \href{samarghandi@edwards.usask.ca}{samarghandi@edwards.usask.ca}} \:, Mehdi Behroozi%
\thanks{Department of Mechanical and Industrial Engineering, Northeastern University, Boston, MA, USA. Email: \href{m.behroozi@neu.edu}{m.behroozi@neu.edu}}}

\maketitle
\vspace{-1cm}
\begin{abstract}
    Supply chain disruption can occur for a variety of reasons, including natural disasters or market dynamics for which resilient strategies should be designed. If the disruption is profound and with dire consequences for the economy, it calls for the regulator's intervention to minimize the impact for the betterment of the society. This paper considers a shipping company with limited capacity which will ship a group of products with heterogeneous transportation and production costs and prices, and investigates the minimum quota regulation on transportation amounts stipulated by the government. An interesting example can happen in North American rail transportation market, where the rail capacity is used for a variety of products and commodities such as oil and grains. Similarly, in Europe supply chain of grains produced in Ukraine is disrupted by the Ukraine war and the blockade of sea transportation routes, which puts pressure on rail transportation capacity of Ukraine and its neighboring countries to the west that needs to be shared for shipping a variety of products including grains, military, and humanitarian supplies. Such situations require a proper execution of government intervention for effective management of the limited transportation capacity to avoid the rippling effects throughout the economy.     
    We propose mathematical models and solutions for the market players and the government in a Canadian case study. Subsequently, the conditions that justify government intervention are identified, and an algorithm to obtain the optimum minimum quotas is presented. 
\end{abstract}

\keywords{Rail Transportation Pricing Optimization; Game Theory; Minimum Quota Regulation; Resilient Commodities Supply Chain}

\section{Introduction}
\label{sec:Intro}
Many societal requirements and special circumstances may necessitate government's limited intervention in free markets. The most recent and major example of such benevolent interventions is the unprecedented policies of the governments around the world in the wake of the COVID-19 Pandemic. The field of transportation and supply chain can be heavily disrupted by crises of this type and can benefit from policymakers' intervention. Supply chain disruptions can occur due to several other forms of events such as a natural disaster in a major port, an obstruction in the Suez Canal, chronic shortages in products that are either not economical to produce or their technology is not available, strikes impacting major suppliers of certain products, and sanctions or blockades against a major producer country of certain products or commodities. Such disruptions could happen in the supply chain of almost all products regardless of how basic or advanced they might be.

It is also possible for these disruptions to happen due to the free markets internal dynamics. These may occur when a purely profit-driven supply chain values certain products more than others while the disfavored products, although non-profitable for the supply chain and transportation companies, are vital for the society. Furthermore, time sensitivity of such situations or perishability of the products may not allow us to wait for the market to regulate itself. Imposing a minimum shipment quota for certain products may be necessary for the betterment of society or resiliency reasons. 
Two interesting cases happened in Canada's \citep{CanadaCase} and North Dakota's \citep{USNYT,USNBCNews} rail transportation systems. Here, we focus on the Canadian case.

In 2013, the Canadian railroad capacity was strained. Grain harvested, which needed to be shipped to the consumption markets by rail, reached an all-time high. Soaring global oil prices and limited pipeline capacity caused the producers to consider shipping oil to the international markets by rail. Meanwhile, an abnormally cold winter stalled the trains. A carry-over of 18 million tonnes of crops was expected in 2014, six million tonnes higher than the average of 12 million. Moreover, oil producers' profit margins were high; hence, they were able to allocate most of the rail transportation capacity to oil and oil-derivatives. This situation urged the Federal Government to ensure that carry-over volumes of grain will be shipped to the markets.

In May 2014, after a series of consultation with various interest groups and stakeholders including farmers, oil producers, various other commodity producers such as chemical and lumber producers, transportation companies, public representatives, etc., the Federal Government passed ``Canada's Fair Rail for Grain Farmers Act'' \citep{Canada2014Act} to reduce backlogs at rail terminals and grain elevators, and to ensure that farmers can transport their grain to the markets during the peak price season. The act states that Canadian National Railway\footnote{\href{https://www.cn.ca/en}{https://www.cn.ca/en}} (CN) and Canadian Pacific Railway\footnote{\href{https://www.cpr.ca/en}{https://www.cpr.ca/en}} (CP) must each ship 500,000 tonnes of grain per week. The described interventions is not unique to the commodities transportation market and may occur in several other circumstances. For instance, in the wake of the disruption caused by the COVID-19 pandemic, in 2021 American government established a Supply Chain Disruptions Task Force to address short-term supply chain discontinuities and shortages for critical products such as semiconductor manufacturing and advanced packaging; large capacity batteries, like those for electric vehicles; critical minerals and materials; and pharmaceuticals and active pharmaceutical ingredients (APIs) \cite{scTaskForce2021}. It also intervened in the operation of Ports of Los Angeles and Long Beach to address transportation bottlenecks to ensure availability of consumer goods in the shelves of U.S. retailers \cite{PortLA2021intervention}. Motivated by the described events, this article examines the minimum quota regulation stipulated on transportation companies, where a group of suppliers are competing over the transportation capacity. This is important because some of these products are critical for the consumers, and their shortage in the market results in severe societal consequences, which necessitates regulator's intervention. Avoiding negative ramifications is so crucial that justifies the regulator's intervention in the market, although it may slightly reduce the suppliers' and the transporter's profits.

This paper assumes that there are several suppliers producing distinct products and commodities. These commodities must be shipped from the point of production to the point of sale or consumption using the same mode of transportation, which, in our case, is rail. Since the number of available trains and shipping containers are limited, the suppliers compete with each other to gain access to transportation. If the profits for one of the products is much higher than the others, the producers will offer dominating transportation prices to the transportation provider, essentially cutting the producers of the other products off the market. 

This was the case in 2013, when oil producers offered much higher prices for rail transportation compared to grain producers, which left the farmers in a weak position. The grain backlog and unavailability of transportation costed farmers between two and four billion dollars. In case one or more of these left-off products are of strategic importance, government will be willing to intervene in favor of the suppliers if they cannot suggest a competitive transportation price to the shipping firms. As a result of the described situation, the Federal Government mandated a quota (5,500 carloads per week at the time) for grain transportation and a penalty for non-compliance through ``\emph{Canada's Fair Rail for Farmers Act}.'' 

Once the regulation was in place, all the stakeholders expressed disappointment. According to grain producers, neither the assigned quota was enough to remove the backlog, nor the penalty was enough to convince the transporters to meet the quota. The farmers believed that the oil producers were able to cover the penalty for the transportation companies. On the other hand, the oil producers accused the government of intervening in the free market of transportation and tilting the ground against the profitable commodity.

Transportation capacity, either defined for people's movement in a transportation network or moving commodities, depends almost completely on node capacity and route capacity. In general, transportation capacity is the \emph{maximum throughput} of a system performed within the permitted constraints \citep{Kasikitwiwat2005}. Therefore, for the rail transportation, capacity could mean the minimum of total car capacity (appropriate for handling a specific product), the rail road capacity, and temporary stocking capacity in warehouses, which also includes loading/unloading capacity. In this paper, rail transportation capacity is measured by the number of carloads of products shipped during a certain amount of time, which In North America, is reported by the transporters on a weekly basis, and is available from various sources\footnote{For example, see the weekly reports of the Association of American Railroads (AAR), accessed on September 5, 2022 at \href{https://www.aar.org/data-center/rail-traffic-data}{https://www.aar.org/data-center/rail-traffic-data}}. When the demand for shipping a product is more than the addressed transportation capacity for a limited amount of time, the suppliers need to stock the excess products and pay the required expenses. However, if the total supply and demand for the products are constantly more than transportation capacity, as was the case for the described Canadian episode, outside (non-market) intervention is required to guarantee some transportation share for the suppliers of all the various products and commodities requiring transportation. Having this in mind, we seek to address the following research questions regarding the public policy and transportation regulation:

\ResearchQuestion{When should the regulators intervene in the market in favor of the producers of one of the commodities?}
\ResearchQuestion{Where intervention is necessary, how should the minimum quota be defined so that the overall outcome benefits the society as a whole?}
\normalfont

This research aims to analyze the dynamics between producers, shippers, and regulators as the main players of the market, carefully and thoroughly, to advance the understanding and descriptive modeling that explain the actions of the market players, and prescribe methods that improve market resilience and response to disruptions. One main direction is to make the impact of regulation of the whole system more effective and efficient, for which the regulator needs to comprehend the market mechanism completely.
In our case, market mechanism consists of the impact of shipping cost and price of a commodity on the commodity's allocated shipping capacity. The objective of this research is to establish guidelines for the regulator in setting a minimum quota on the shipping amount for the commodities that will be otherwise cut off from the market. Therefore, the pricing policies of the suppliers and the transportation firm will be investigated, and consequently, the amount of various commodities being shipped will be checked. This clarifies the market reaction to introducing minimum quota regulations and the impact of such interventions on the commodity transportation amounts. Hence, the models proposed in this research prove extremely beneficial to the regulator in following the interests of the producers and the public simultaneously. 

The problem under consideration in this research is as follows: there exists a market with multiple players: one transportation company with limited transportation capacity, several suppliers which need to ship their products to the point of consumption, and a regulator which overlooks the transportation and consumption market for the sake of the consumers. The goal of the suppliers and the shipping company is to maximize their own profits; the suppliers through offering competitive prices to the shipping company that lead to gaining access to more transportation capacity, and the shipping company by allocating its capacity to those suppliers which prove most profitable. On the other hand, the goal of the regulator is to maximize social utility by ensuring a minimum level of availability for critical products. We want to determine minimum quotas for various groups of commodities, scrutinize the impact of regulator's decisions on the shipment amounts, and investigate the pricing policies of the suppliers and the transportation firm, such that social utility is maximized without significantly sacrificing the profit of the suppliers and the shipping company.

To the best of our knowledge, this is the first paper that discusses government intervention in a transportation market in the form of minimum quota, where there is a vertical competition between the shipping company and the suppliers, and there is a horizontal competition between the suppliers. In addition, real data from the Canadian market is used to clarify the implementation of the methods and algorithms discussed in this paper.

The rest of the paper is organized as follows. Section \ref{sec:lit} is devoted to the review of the related literature. Section \ref{Model} presents the mathematical models for the suppliers and the shipping company. The next step in defining a successful market intervention policy is investigating the reaction of the market to introducing minimum quota regulations; this will be discussed in Section \ref{sec:Regulation}. A numerical study based on real data from Canadian transportation market is presented in Section \ref{Numerical}. The paper is summarized and concluded in Section \ref{Conclusion}. 

\section{Literature Review}
\label{sec:lit}
It is often the case that resources required by a population are limited, and an efficient management system is essential to tackle the problem of resource scarcity. Limitation in resources happens when a company deals with high demand and the supplies are low. Facing high demands and having access to low supplies is one of the several complications of transportation systems as well; hence, an efficient transportation management system is required to maximize companies' revenues. Matching supply and demand in transportation realm can be done using a variety of approaches; namely, changing the transportation pricing scheme, and changing the regulations governing the transportation market. Both of these approaches are of interest for this research. We review the literature from both perspectives. \vspace{-5pt}

\paragraph{Transportation pricing and capacity allocation:} Designing an efficient pricing model plays an important role in the profitability and success of the transportation systems. Any transportation pricing effort must consider the structure of the market. Such studies may also need to make certain assumptions regarding some external cost factors such as congestion and scarcity cost, accident cost, air pollution and noise cost, energy dependency cost, and climate change \citep{maibach2008handbook}. For instance, \cite{mcauley2010external} developed a model that compared the external costs of shipping non-bulk products using rail or road as competing modes between Australian state capitals by analyzing external costs of freight transportation including pollution, accidents and congestion. \cite{delucchi2011external} conducted the same project as \cite{mcauley2010external}, but considered all modes of transportation available in the United States. Similarly, \cite{demir2015selected} conducted a review on freight transportation pricing and found that both external and internal factors have an impact on the transportation price. Below, we review the transportation pricing articles most relevant to this study.

\cite{zhou2009pricing} proposed a mathematical model to optimize the strategies under a monopoly and duopoly market structure. \cite{toptal2011transportation} studied transportation pricing in a market where a truckload and a less than truckload carrier operate, and showed that cooperation between the two carriers leads to better pricing schemes and saving opportunities. \cite{azadian2018service} developed a mixed-integer non-linear programming model to determine transportation prices based on grouping of service locations in similar geographical areas as an external cost factor, and proposed two decomposition-based approaches to solve their model. \cite{chang2018green} presented a road pricing model to find a price-optimal and environmentally-friendly transportation system considering electronic toll collection systems, as another external cost factor. In a similar work, \cite{gu2018optimal} used both time and travel distance to calculate the toll in a transportation system and noticed that imposing distance toll results in uneven distribution of congestion in the travel network as the users choose the shortest path to avoid distance tolls. 
Transportation pricing has been used in many different application areas. For example, \cite{chen2018pricing} solved a pricing problem for a last-mile transportation system and validated their model by implementing their approach for a case in Singapore. 
\cite{johari2019modeling} developed a method for cordon pricing policy optimization models in a multi-modal traffic network and showed that cordon pricing results in increased usage of public transport system over private transportation means. 
 \cite{woo2020can} analyzed price elasticity in Hong Kong's public transportation system to price-manage its ridership; they noticed that reduction of public transportation fares will not increase public transportation ridership. Instead, elements such as increasing accessibility of public transportation and restricting the usage of private transportation must be taken into consideration. \vspace{-5pt}

Transportation pricing problems have been also studied in the context of game theory. 
For example, \cite{mozafari2015dynamic} studied a dynamic pricing problem for freight carriers who compete with each other in an oligopolistic freight network and formulated the problem as a discrete-time dynamic Nash equilibrium and then incorporated a generalised Branch-and-Bound procedure into a decomposition algorithm to solve the problem.
\cite{peng2016stable} developed a price game mechanism based on Gale-Shapley algorithm to solve a stable vessel-cargo matching problem in the dry bulk shipping market and showed that the proposed price game mechanism provides a bigger surplus for disadvantaged participants. The goal of increasing surplus for the disadvantaged participants has some similarities with our approach. However, we take a different strategy to help the disadvantaged shippers to ship their product in a competitive market by imposing a minimum capacity quota, which will be discussed in Section \ref{sec:Regulation}. In another work similar to ours, \cite{tamannaei2021game} considered a case where the customers of a commodity are able to choose between direct and intermodal transportation systems, called DTS and ISPs. In contrast to our work, the prices for the DTS are readily available, while the cost of using the ISPs are suggested by the transportation companies. They proposed a non-cooperative game-theoretic approach based on Stackelberg leader-follower competition to find the equilibrium prices in this multi-level multi-modal freight transportation pricing problem, and find the equilibrium decisions between transporters and customers using mixed-integer linear programming models. \cite{adler2021review} provides a good review on transport market modeling using game-theoretic approaches. \vspace{-5pt}

In another stream of works, efficient allocation of freight transportation capacity is studied. For example, \cite{amaruchkul2011air} considered the cargo capacity allocation problem of an air-cargo carrier which uses revenue management to distinguish between different classes of freight forwarders such that their allocated capacity maximized the cargo carrier's revenue, and proposed a dynamic programming and two heuristics to solve small and large instances of their model, respectively. In a more recent instance, \cite{taherkhani2022tactical} focused on tactical capacity planning of a transportation system with several stakeholders in which revenue management is implemented, and proposed a model to develop a transportation plan to maximize the profit. In addition, they investigated the computational efforts required to solve their model and investigated the effect of changes in system parameters on the results. 
Transportation capacity can also become challenging when the limiting factor is the available railroad infrastructure, especially when it has to be shared between freight and passengers. In many places, passenger trains are preferred to freight trains when railway capacity becomes scarce. \cite{bablinski2016game} aimed at improving the competitiveness of freight train operators compared to passenger trains using capacity allocation techniques and developed a simulation game based on real data from Brighton Main Line in England with the goal of improving social welfare and equity between different operators, and proposed a number of recommendations to the practiced British railroad policies. 
\cite{xu2022optimize} assumed that high-speed rail systems can carry both passengers and freight, and proposed a revenue management scheme with deterministic and stochastic demand scenarios for such a shared transportation system with the goal of maximizing the rail operator's profit.

\paragraph{Transportation regulation:} Changing transportation regulations deeply influences the well-being of a society. Combined with the prominence and ubiquity of transportation, this topic has been an interesting field for scholars, practitioners, and policymakers. 
For instance, \cite{aschauer2010time4trucks} advised that imposing congestion time regulations on transportation companies and retailers can lead to a reduction of truck traffic from urban roads during rush hours through simulation and rescheduling, which will lead to reduction of traffic congestion and carbon emission, and postponement of capital expenditure for road construction. \cite{goel2014hours} proposed an optimization model to assess the impact of ``hours of service'' regulation on the bottom line of a freight transportation company who operates a fleet of vehicles when the business hours of the customers are taken into consideration. Their analysis showed that hours of service regulations in the European Union maximizes safety, whereas Canadian regulations reduce the costs of transportation.
Regulations related to carbon emission and sustainability measures have also been a subject of research.
\cite{quak2008sustainability} reviewed the consequences of most commonly used sustainability-oriented regulations on economical, environmental, and social costs of retail freight distribution and proposed methods of reducing the burden of such regulations on retail distribution.
\cite{wei2020research} justified that logistics providers must optimize their routes while considering the low carbon emission regulations, and concluded that the existence of high grade roads such as expressways and shorter routes are substantial in achieving the environmental goals.
\cite{tiwari2021freight}, motivated by an Indonesian real case, considered a freight consolidation and containerization problem under a carbon tax regulation scenario and various carbon footprints schemes to minimize transportation costs while lowering the amount of carbon emission from maritime and land transportation modes, and developed a mixed-integer program to solve their model.
\cite{galkin2022assessment} proposed a number of tax schemes for freight transportation in urban locations to reduce the usage of certain vehicles, and compared the impact of their tax schemes with strictly forbidding these forms of transportation in encouraging the adoption of more preferred forms of transportation such as green options.

Moreover, in many countries there exist plans for increasing the share of railroads in freight transportation and incentives for a move towards electric vehicles. For example, \cite{demirci2017designing} investigated the government's intervention strategies for increasing the demand for public-interests goods through providing incentives and rebates for consumers, and used real data from California's electric vehicle market for model validation and verification.
The goals of regulations that intend to impose certain changes to the transportation market could also be achieved as unintended consequences of other regulations and policies. For example, \cite{gnap2018possible} reasoned that the hours of service regulations for truck drivers in the European Union will shift freight transportation from the trucking industry to rail carriers, mostly due to the shortage of drivers, age structure of the drivers, and possibility of freight theft. Motivated by such policies, \cite{van2014development} focused on the regulations that aim to increase the share of rail carriers in long-haul freight transportation in the European Union, to measure the impact of various development scenarios. Their study scrutinized the role of regulations on a range of future outcomes of the market structure, from monopolies to oligopolies, among Belgian rail carriers. 
In many countries, there are regulations in place to prevent monopolistic behavior. Price cap regulatory, as a mechanism to control prices in network industries, is an example of the aforementioned regulations \citep{beesley1989regulation, brennan1989regulating, isaac1991price}. In such a regulatory regime, cap is fixed to the average proposed prices of the regulated company and may change through time because of inflation \citep{breton2012transportation}. \cite{kang2000consumers} studied the effect of reducing the price cap on customers' welfare, and realized that tightening the price cap increases consumer welfare when demand is independent. On the contrary, if demand is inter-dependent, tightening the price cap decreases customer welfare.  
\vspace{-5pt}

One can notice that regulation of freight transportation has been studied from various angles and in diverse jurisdictions. However, it appears that the impact and significance of regulating freight transportation capacity, especially on supply chain of critical products has not gained much attention.

\paragraph{Research gaps and contributions of this paper:}  The above literature review exposes important facts and gaps in the field of transportation research relevant to this paper. Below, we highlight some of these gaps pertinent to the contributions of this study, which will be summarized consequently.
\begin{enumerate}
    \item There are a variety of reasons or situations where the governments may decide to intervene in the activities of the free market for the greater good of the society. Furthermore, the activities of lobbyists that pressure the government to intervene in the market in favor of their interest groups is another trigger for such interventions. The field of transportation, given its importance and pervasiveness, is no exception. However, the literature that focuses on the necessity of enacting minimum transportation quota regulations has not been studied. This article is the first attempt to research this matter.
    \item Transportation is a fundamental activity for the economic prosperity and societal functioning of a country and since it involves physical activities it is highly impacted by the geography of the operations. As such, most studies with a focus on transportation define their parameters based on a certain geographical area's limitations, potentials, laws and regulations. Our review of the literature revealed that a similar study with detailed analysis and modeling does not exist in the context of Canadian transportation system.
    \item The need to study the behavior and reactions of transportation market participants to certain laws and regulations is deeply felt. This paper fills this gap for the case of minimum quota regulations in the rail transportation market.
    \item The existence of both vertical and horizontal competition impacted by new regulations in the transportation market has not gained much attention in the literature. To the best of our knowledge, this paper is the first attempt to analyze the dynamics of such competitions in the context of imposing transportation capacity quota for critical products.
\end{enumerate}

We consider the case of minimum quota regulation in a transportation market with vertical competition between the shipping company and the suppliers, and horizontal competition between the suppliers, which to the best of our knowledge, is the first effort for dealing with this problem. To do this: 
\begin{enumerate}
    \item We formulate the non-cooperative game between the suppliers' and the shipping company as a Stackelberg game.

    \item We formulate and solve the shipping company's problem (the follower's problem) as a bounded continuous knapsack problem.

    \item We identify that each supplier solves a bi-level bounded continuous knapsack problem to maximize its own profits. Furthermore, we model the horizontal competition between all the suppliers as a Bertrand Competition with Capacity Constraints, where the suppliers are categorized into distinct classes based on inputs such as the production amount, market price, inventory and transportation cost of their products. We prove that Mixed Strategy Nash Equilibrium (MSNE) exists for the suppliers and find the cumulative probability function of MSNE transportation price for each supplier.

    \item We find the necessary condition for assigning minimum quota to each supplier, and formulate the regulator's problem as a Social Utility Maximization Model using the concept of Maximum Revenue Entitlement.

    \item Finally, we use real data from the Canadian market to showcase how the methods described in this paper can be used to solve real-world problems.
\end{enumerate}

Although, this research fills some important gaps in the literature, similar to any other research, it involves some limiting assumptions. For instance, we only consider one rail transporter, whereas it is possible to have multiple rail transporters in an area. Furthermore, the suppliers rely merely on the rail transportation mode and the availability of different transportation modes is not considered. Moreover, some aspects of the the post-regulation environment is omitted from the analysis. For example, the mechanism for preventing the shipping company from evading the imposed minimum quotas and refraining from transporting the low-profit products is not considered. Similarly, the mechanism for ensuring the fairness of the prices offered by less competitive suppliers who receive a minimum quota is not included in our analysis.

In the next section, we present the used notations, as well as the mathematical models that describe the actions of the shipping company and suppliers. Subsequently, the regulatory implications are discussed.

\section{Mathematical Model}
\label{Model}
Assume that there are $|I|$ types of products to be transported by rail; we use $i\in I$ to index these products. Also, assume that each product $i$ is produced by a unique producer, which is indexed the same as the products. Moreover, consider that there are $D_i$ units of product $i$ available for shipping. The market-driven sale price of each unit of product $i$ is $S_i$. 
The cost of transporting each unit of product $i$ for the transportation company is $c_i$. The offered shipping price of supplier $i$ to the transportation company for shipping each unit of product $i$ is $p_i$. 
In this paper, the sale price ($S_i$) is assumed to be an exogenous parameter determined by the bigger commodity market consisting of all or most producers and customers of that commodity, potentially at a regional or even global stage in commodity exchanges. Furthermore, we assume that the available units of a product ($D_i$) is not impacted by the sales price ($S_i$) or the shipping price ($p_i$) since the considered model is a single period problem, whereas the mutual impact of market price and production amount appears in future periods.

The shipping company decides on the quantity $u_i$ of product $i$ to ship based on the cost of transportation and the offered price of transportation. Note that the volume units of all products must be unified and all the other product-related specifications must be adjusted accordingly before solving the model. Table \ref{tab:notation} describes the notations used in this paper.

\begin{table}[t]
    \centering
    \caption{Notations used in the paper.}
    \label{tab:notation}
    \begin{tabular}{cl} \Cline{1.5pt}{1-2}
        Parameters \& Variables & Definition  \\ \Cline{1.5pt}{1-2}
        $T$ & Total transportation capacity \\
        $\mbox{MRE}_{i}$ & Maximum revenue entitlement of the shipping company from product $i$ \\
        $\beta_i$ & Financially scaled social importance score of one unit of product $i$ \\
        $D_i$ & Total production amount (ready for shipment) of product $i$ \\
        $S_i$ & Market price of product $i$ \\
        $g_i$ & Inventory cost per unit of unsold product \\
        $c_i$ & Transportation cost per unit of product $i$ \\
        $p_i$ & Transportation price per unit of product $i$ \\
        $u_i$ & Transported amount of product $i$ (allocated space) \\
        $L_i$ & Minimum quota of product $i$ \\
        \Cline{1.5pt}{1-2}
    \end{tabular}
\end{table}

As discussed before, the suppliers compete for the allocated capacity by suggesting better prices to the shipping firm. Upon receiving the suggested prices, the shipping firm decides how much capacity to allocate to each product. Since the suppliers compete over the transportation capacity, and none of them has power or first-move advantage over the others, the competition between the suppliers can be best modeled by a multi-player game for which the Nash equilibrium concept can be applied. Afterward, the non-cooperative decision making between the suppliers and the shipping firm fits the Stackelberg game structure, in which the suppliers are called the ``leaders'', and the shipping firm is denoted as the ``follower''. To solve the described Stackelberg game, we start by solving the follower's problem. Figure \ref{fig:GameStructure} illustrates the framework of the game and the regulator's problem. \vspace{-5pt}

\begin{figure}[h]
    \centering
    \includegraphics[width=\textwidth]{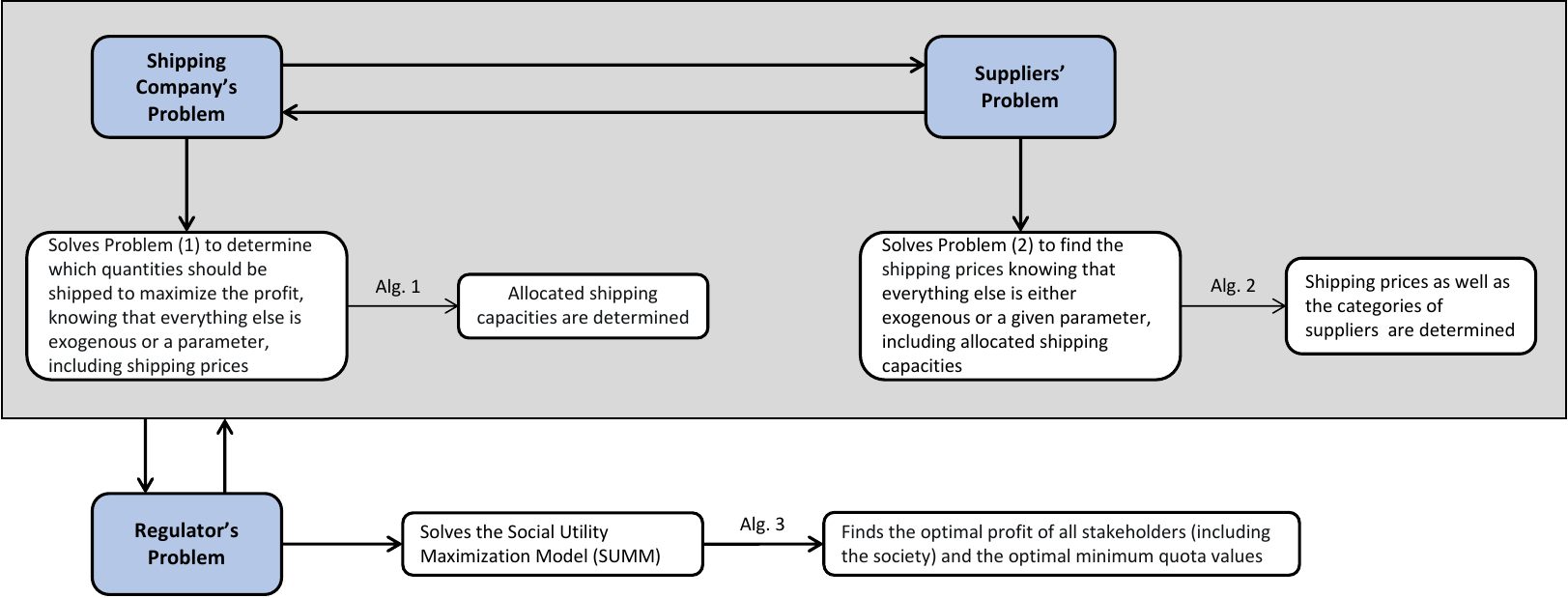}
    \caption{Game structure and the regulator's position}
    \label{fig:GameStructure}
\end{figure}

\subsection{The Shipping Company's Problem}
\label{sec:shipping}
\vspace{-5pt}The shipping company maximizes its profit by solving: \vspace{-10pt}
\begin{eqnarray}
\label{model:shipping}
   \maximize_{\bm{u}} & \underset{i\in I}{\sum}\, u_i(p_i-c_i)  & \qquad \mbox{s.t.}  ~\\
    &  \underset{i\in I}{\sum} \, u_i \; \leq \;  T \,, &   \nonumber \\
    &  0 \; \leq u_i \; \leq \;  D_i \,,   &  \qquad \forall i\in I \nonumber 
\end{eqnarray}
where $\bm{u}=(u_1,\ldots,u_{|I|})$. Note that $p_i$ in this problem is assumed to be a parameter and its value is determined by the supplier as will be discussed in Section \ref{sec:suppliersProblem}. The objective function is a linear combination of transportation profits for all suppliers. Since transportation amount, $u_i$, can be any value below $D_i$, and because partial transportation is possible, the transportation company starts from the supplier with the highest transportation margin $(p_i-c_i)$. In other words, the shipping company is solving the well-known bounded continuous knapsack problem, which garners the solution procedure of Algorithm \ref{alg:1} that describes the iterative procedure the shipping company can employ to allocate the transportation capacity to suppliers such that its profit is maximized.

\begin{algorithm}[htb]
\protect\caption{\label{alg:1}Algorithm ${\sf CapacityAllocator}$ takes the input parameters of the problem and allocates the transportation capacity to the products.}
\SetAlgoLined
\BlankLine
\KwIn{The set of products $I$, production amount $D_i$, transportation price $p_i$, and transportation cost $c_i$ for each product $i\in I$, and total transportation capacity $T$.}
\KwOut{Transportation capacity $u_i$ allocated to each product $i\in I$.}
\BlankLine
\tcc{**********************************************************************************}
Sort $(p_i-c_i), \; i\in I$ in descending order. If a tie occurs, break it randomly\;
Let $Z=\{p_z-c_z,p_{z-1}-c_{z-1},...,p_1-c_1\}$ be the sorted set\;
$t \leftarrow z$\;
$u_i \leftarrow 0$, $\forall i\in I$\;
\While{$T>0$ \AND $t>0$}{
$u_k \leftarrow \min\{T,D_t\}$\;
$T \leftarrow T-u_t$\;
$t \leftarrow t-1$\;
}
\KwRet{$u_i\,, i\in I$}\;
\end{algorithm}

Note that, in each iteration, for any value of the remaining capacity, the solution falls into one of the five cases addressed in Table \ref{tab:bestStrategy}. The described relationships of Table \ref{tab:bestStrategy} are depicted in Figure \ref{fig:my_label}. If the solution procedure starts at $C_1$, it remains in $C_1$, i.e., all suppliers get their maximum required capacity. If current unassigned capacity falls into $C_2$ at any iteration, the next iteration may fall into $C_2$ to $C_5$. Once in $C_4$, the possible next steps are $C_3$, $C_4$, and $C_5$. Finally, $C_3$ and $C_5$ are terminal cases, as was $C_1$. 
\begin{figure}[t]
    \centering
    \includegraphics[trim={3cm 0 3cm 0},clip,width=0.5\textwidth]{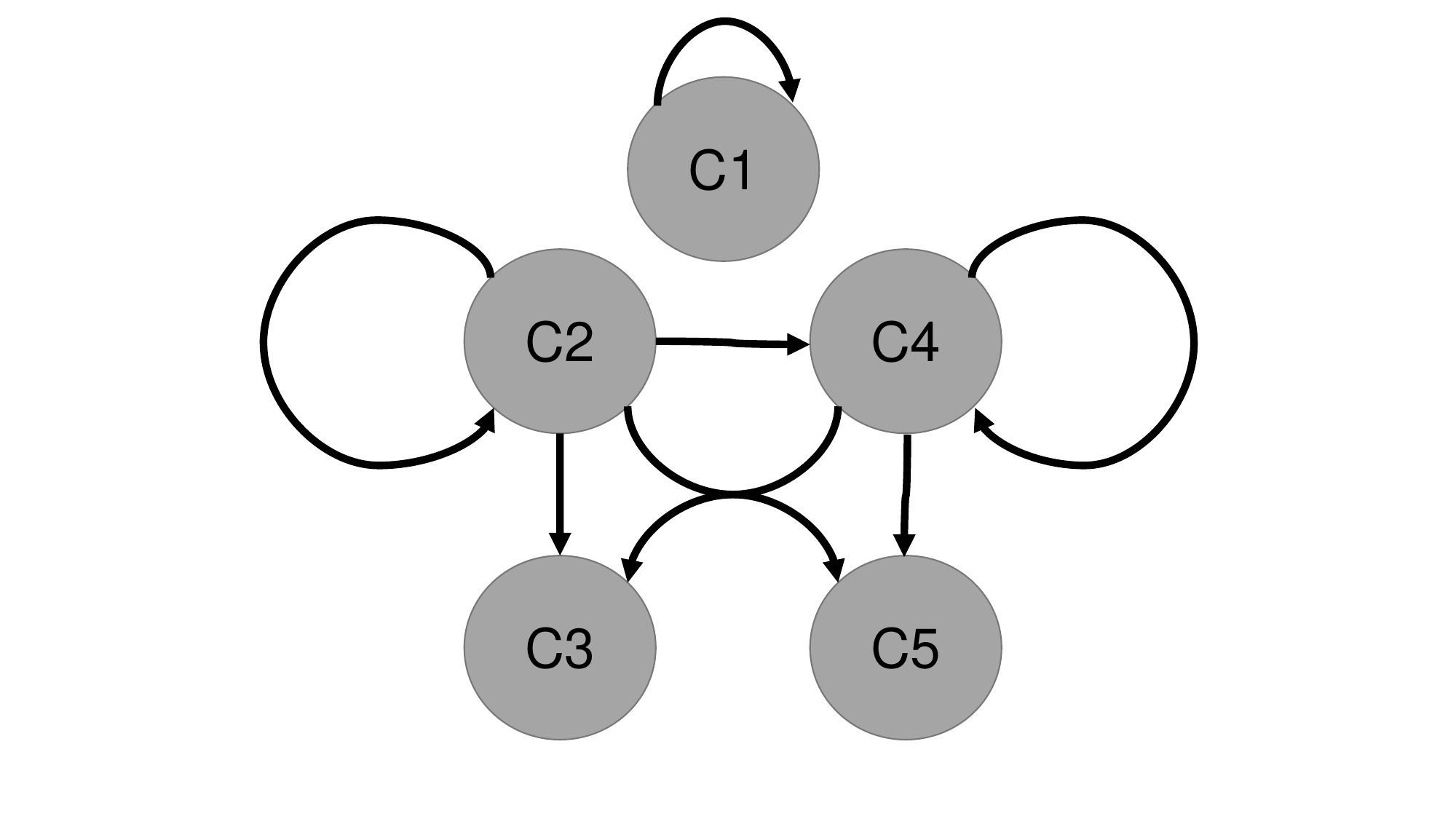}
    \caption{Interconnection of all cases.}
    \label{fig:my_label}
\end{figure}
~\\
\begin{table}[h]
    \centering
    \caption{Transportation Company's Optimal Strategy.}
  \scalebox{0.85}{
    \begin{tabular}{c|lll|lll} \hline
        Case & \multicolumn{3}{c|}{Conditions} & \multicolumn{2}{c}{Solution} \\ \hline
        $C_{1}$ & $\sum_i D_i\le T$ & & &  & $u_i=D_i$ &  \\[1ex]
        $C_{2}$ & $\max\{D_i\}\le T \le \sum_i D_i$, & $\max\{p_i-c_i\}=p_{i*}-c_{i*}$ & &  & $u_{i*}=D_{i*}$, & $\sum_{I-\{i*\}} u_{i}=T-D_{i*}$  \\[1ex]
        $C_{3}$ & $\min\{D_i\}\le T\le \max\{D_i\}$, & $\max\{p_i-c_i\}=p_{i*}-c_{i*}$, & $D_{i*} \ge T$ &  & $u_{i*}=T$, & $u_{-i*}=0$ \\[1ex]
        $C_{4}$ & $\min\{D_i\}\le T\le \max\{D_i\}$, & $\max\{p_i-c_i\}=p_{i*}-c_{i*}$, & $D_{i*} \le T$ &  & $u_{i*}=D_{i*}$, & $\sum_{I-\{i*\}} u_{i}=T-D_{i*}$ \\[1ex]
        $C_{5}$ & $T\le \min\{D_i\}$, & $\max\{p_i-c_i\}=p_{i*}-c_{i*}$ & &  & $u_{i*}=T$, & $u_{-i*}=0$ \\ \hline
    \end{tabular}}\label{tab:bestStrategy}
\end{table}

Consequently, the following statements can be made about the suppliers' competition:

\begin{itemize}
    \item \textbf{Case 1} represents a situation where there is enough transportation capacity for all suppliers. Therefore, this scenario is omitted from the analysis for it does not require government intervention. 
    \item \textbf{Cases 2} and \textbf{4} are intermediate phases, meaning that there are zero or more suppliers that transport all their products. These cases will be called group 1 ($G_1$).
    \item \textbf{Cases 3} and \textbf{5} can be both intermediate and terminal phases, meaning that there is zero or one supplier whose product is partially moved. These states are called group 2 ($G_2$).
    \item The rest of the suppliers cannot transport their products, and are called group 3 ($G_3$).
\end{itemize}

Note that all parameters are adjusted to represent the same unit of volume or weight. 
After calculating the best expected responses of the follower (the shipping company), we proceed with solving the suppliers' (leaders) problem.

\subsection{The Suppliers' Problem}
\label{sec:suppliersProblem}
Supplier $i \in I$ has produced $D_i$ units of product (in terms of volume or weight) and intends to ship them by train to a certain destination to be sold at price $S_i$. Supplier $i\in I$ solves the following bi-level bounded continuous knapsack problem:
\begin{eqnarray}
\label{model:supplier}
    \underset{p_i \geq 0}{\maximize} & v_i(S_i-p_i)-g_i(D_i-v_i)\,, & \qquad \mbox{s.t.}  ~\\
    & v_i = \bm{v}^T \bm{e}_i\,, & \nonumber \\
    & \bm{v} \in \underset{\bm{u}}{\argmax} \left\{\underset{i\in I}{\sum}\, u_i(p_i-c_i): \underset{i\in I}{\sum} \, u_i \; \leq \;  T, \;\; 0 \; \leq u_i \; \leq \;  D_i\right\}\,,  &  \nonumber
\end{eqnarray}
where $\bm{e}_i$ is the $i$th column of the identity matrix.
The first term in the objective function is the production profit, and the second term is the inventory/storage cost in case the shipment quantity is below production quantity $D_i$. Each supplier $i$ suggests a transportation price $p_i$ at the same time as the other suppliers to achieve the maximum transportation capacity of $D_i$.

Looking into the shipping company's solution procedure (Algorithm \ref{alg:1}), assuming the same transportation costs, the supplier with the highest suggested transportation price will be the first to win the transportation capacity. Furthermore, if there are two suppliers and $T \leq D_i\,, \; i=1,2$, the winner takes all the capacity. Therefore, the competition between the suppliers can be best modeled as \textit{Bertrand Competition with Capacity Constraints}.

\begin{remark}
The supplier (producer) has already made a decision on the production amount $D_i$. In other words, if the supplier wins a transportation capacity below $D_i$, the production quantity cannot be adjusted accordingly. The excess production will be stored at the cost of the supplier.
\end{remark}

As previously discussed, there are three groups of suppliers: a) suppliers that will transport all of their product ($G_1$); b) supplier that will transports part of their product ($G_2$); and, c) suppliers that will not transport any of their product ($G_3$). It is clear that within the range of offered transportation prices, $p_i$, there will be dividing points that separate these groups. Those prices are also indifference points, since, for example, a supplier offering a price equal to the dividing point between $G_1$ and $G_2$ will be indifferent in choosing either group because its gain will be the same either way. Proposition \ref{prop:diff-point} utilizes this concept to investigate pricing strategy of suppliers.

\begin{proposition}\label{prop:diff-point}
The indifference transportation prices between $G_1$ and $G_2$, $G_2$ and $G_3$, and $G_1$ and $G_3$, are the same, which is $S_i+g_i$. 
\end{proposition}
\begin{proof}
The indifference point between $G_1$ and $G_2$ is where supplier $i$ gains the same whether $u_i=D_i$ or $u_i=\hat{D}_i$, where $\hat{D}_i<D_i$.
The price cap is obtained by having the profit of partial transportation less than the profit of full transportation.
\[
\underbrace{\hat{D}_i(S_i-p_i)-g_i(D_i-\hat{D}_i)}_{\text{Transporting $\hat{D}_i$}}\le\underbrace{D_i(S_i-p_i)}_{\text{Transporting $D_i$}} \quad \Longleftrightarrow \quad p_i\le S_i+g_i
\]
Similarly, the indifference point between $G_1$ and $G_3$ is where supplier $i$ gains the same, whether $u_i=D_i$ or $u_i=0$:
\[
\underbrace{-g_iD_i}_{\text{Transporting $0$}}\le\underbrace{D_i(S_i-p_i)}_{\text{Transporting $D_i$}} \quad \Longleftrightarrow \quad  p_i\le S_i+g_i
\]
Finally, the indifference point between $G_2$ and $G_3$ is where supplier $i$ gains the same, whether $u_i=\hat{D}_i < D_i$ or $u_i=0$:
\[
\underbrace{-g_iD_i}_{\text{Transporting $0$}} \le \underbrace{\hat{D}_i(S_i-p_i)-g_i(D_i-\hat{D}_i)}_{\text{Transporting $\hat{D}_i$}} \quad \Longleftrightarrow \quad p_i\le S_i+g_i
\]
\end{proof}
The following corollaries directly result from Proposition \ref{prop:diff-point}.
\begin{corollary}
The maximum offered price is independent from the total available capacity and production amount.  
One can conclude that the supplier(s) with the highest $S_i+g_i-c_i$ are guaranteed to transport non-zero quantity of products.
\end{corollary}
\begin{corollary}
The maximum price that supplier $i$ suggests is $S_i+g_i$, above which the supplier either prefers to refrain from transporting their product (staying at $G_3$), or is unwilling to offer a higher price to guarantee full shipment of its product (staying at $G_2$). 
\end{corollary}
The insight Proposition \ref{prop:diff-point} provides, together with the solution algorithm presented for the shipping company aid in identifying further details about different types of the suppliers; namely, $G_1$ to $G_3$. The procedure to identify these groups is described in Algorithm \ref{alg:2}, which initially categorizes the suppliers into the three addressed types\footnote{We later discuss that transfers between $G_1$ and $G_2$ may be possible.}.

\begin{algorithm}[htb]
\protect\caption{\label{alg:2}Algorithm ${\sf SupplierGrouper}$ takes the input parameters and categorizes suppliers into groups.}
\SetAlgoLined
\BlankLine
\KwIn{The set of suppliers $I$, production amount $D_i$, market price $S_i$, inventory cost $g_i$, and transportation cost $c_i$ for each product $i\in I$, and total transportation capacity $T$.}
\KwOut{Group category for each supplier $i\in I$.}
\BlankLine
\tcc{**********************************************************************************}
$G_1,G_2,G_3 \leftarrow \{\}$\;
$m \leftarrow |I|$\;
$J \leftarrow I$\;
$p_i \leftarrow S_i+g_i\,, \; i\in I$\;
Let $\bm{u} = {\sf CapacityAllocator}(D_1,...,D_m,\bm{p,c},T)$\;
\While {$T>0$ \AND $m>0$}{
    Set $l = \underset{i \in J}{\argmax}\{S_i+g_i-c_i\}$. If a tie occurs, break it randomly\;
    \eIf {$D_l \leq T$}{
        $T \leftarrow T-u_l$\; 
        $G_1 \leftarrow G_1 \cup \{l\}$\;
    }{
        $T \leftarrow 0$\; 
        $G_2 \leftarrow G_2 \cup \{l\}$\;
    }
    $m \leftarrow m-1$\;
    $J \leftarrow I\backslash(G_1 \cup G_2)$\;
}
$G_3 \leftarrow J$\;
\KwRet{$G_1,G_2,G_3$}\;
\end{algorithm} 

\begin{proposition}
There is exactly one member in $G_2$ unless the production quantities of all $G_1$ members sum up to $T$ in which case $G_2$ remains empty.
\end{proposition}
\begin{proof}
Assume there are more than one supplier, namely $k$ and $\ell$, in $G_2$. First consider the case where $S_{k}+g_{k}-c_{k} < S_{\ell}+g_{\ell}-c_{\ell}$. Then, supplier $\ell$ can offer a higher price and absorb either enough capacity to move from $G_2$ to $G_1$ and thus leaving supplier $k$ as the sole supplier in $G_2$, or consume the entire remaining capacity in which case $T=\sum_{i\in G_1} D_i$ and $G_2$ will be empty. If $S_{k}+g_{k}-c_{k} = S_{\ell}+g_{\ell}-c_{\ell}$, then  one of them is randomly selected by Algorithm \ref{alg:2} and the same logic will hold. This contradicts the assumption of having more than supplier in $G_2$.
\end{proof}
A numerical example of Algorithm \ref{alg:2} follows.

\textbf{Numerical Example:} For simplicity, assume $c_i=c\,,\; \forall i\in I$. Let $|I|=4$ and $T=100$. There are four suppliers with $D_1=40$, $D_2=30$, $D_3=50$, and $D_4=50$ units of product to ship. The transportation cost for all suppliers is the same. Assume that the prices at point of sale $(S_i)$ are $S_1=10$, $S_2=9$, $S_3=12$, and $S_4=6$. The solution procedure is:
\begin{enumerate}
    \item[] \textit{Iteration 0:} $J = I$; $m = 4$. \vspace{-0.25cm}
    \item[] \textit{Iteration 1:} $l = \underset{i \in J}{\argmax}\{S_i+g_i\} \Longrightarrow l=3$; $T=100>D_3=50 \Longrightarrow G_1=\{3\}$; $T=T-D_3=100-50=50$; $u_3=50$; $m=3$; $J = \{1,2,4\}$. \vspace{-0.25cm}
    \item[] \textit{Iteration 2:} $l = \underset{i \in J}{\argmax}\{S_i+g_i\} \Longrightarrow l=1$; $T=50>D_1=40 \Longrightarrow G_1=\{3,1\}$; $T=T-D_1=50-40=10$; $u_1=40$; $m=2$; $J = \{2,4\}$;. \vspace{-0.25cm}
    \item[] \textit{Iteration 3:}  $l = \underset{i \in J}{\argmax}\{S_i+g_i\} \Longrightarrow l=2$; $T=10<D_2=30 \Longrightarrow G_2=\{2\}$; $T=0$; $u_2=10$; $m=1$; $J = \{4\}$. \vspace{-0.25cm}
    \item[] \textit{Iteration 4:} $T=0$;  $u_4=0$.
\end{enumerate}

The next step is to discuss the possible pricing strategies of the suppliers. Proposition \ref{prop:diff-point2} expands Proposition \ref{prop:diff-point} and discusses the price cap for each group.

\begin{proposition}\label{prop:diff-point2}
    The optimal suggestible transportation price for supplier $\ell\in G_1$ is $p^*_{\ell} = \max_{i\in G_2} \{S_i + g_i\}$ and for supplier $\ell\in G_2$ is $p^*_{\ell} = \max_{i\in G_3} \{S_i + g_i\}$. If $G_2$ is empty we will have $p^*_{\ell} = \max_{i\in G_3} \{S_i + g_i\}$ for any supplier $\ell \in G_1$.
\end{proposition}
\begin{proof}
     By Proposition \ref{prop:diff-point} we have $p_i \leq S_i + g_i$. Let the maximum suggestible price for suppliers in $G_2$ be  
    $p_{G_2}^{\max}$, i.e., $p_{G_2}^{\max} = \max_{i\in G_2} \{S_i + g_i$\}. Then, $G_1$ suppliers are guaranteed to transport their full production amount for any price above $p_{G_2}^{\max}$. On the other hand, the supplier in $G_2$ reaches zero profit by offering $p_{G_2}^{\max}$, meaning that the price it is going to offer is $p_{G_2}^{\max}-\epsilon$, for some $\epsilon > 0$. Consequently, $G_1$ suppliers have no gain in offering transportation prices more than $p_{G_2}^{\max}$. With the same token, the optimal suggestible price for suppliers in $G_2$ is $p_{G_3}^{\max}$. Finally, if $G_2$ is empty, suppliers will either have a full shipment or they will not ship their products at all. Therefore, the optimal transportation price for suppliers in $G_1$ is the maximum price that the suppliers in $G_3$ are willing to pay, i.e., $p_{G_3}^{\max}$.
\end{proof}

The next step is to lay the foundations of obtaining the equilibrium. It was discussed that the Nash game between the suppliers is a representation of Bertrand Competition with Capacity Constraints. Because of the discontinuity of the utility function, this type of problem is known not to admit Pure Strategy Nash Equilibrium (PSNE). 
Let $P_i$ be the set of pure strategies of player $i$, i.e., its set of possible prices, and let the function $\pi_i:P=\prod_{i\in I} P_i \rightarrow \mathbb{R}$ be the objective (utility) function of player $i$ that has the set of all possible prices for all players as its domain and the set of real numbers as its range. Then, the objective function of a supplier $i$  
in $G_1$ for a specific strategy $p_i$ is:
\begin{equation}
  \pi_i(p_i,p_{-i}) = D_i(S_i-p_i) \,, 
\label{eq:OFG1}
\end{equation}
where $\pi_i(p_i,p_{-i})$ is the profit of player $i$, when all the other players keep their policy unchanged. 
Similarly, the objective function for a supplier $i$ in $G_2$ is:
\begin{equation}
  \pi_i(p_i,p_{-i}) = \hat{D}_i(S_i-p_i)-g_i(D_i-\hat{D}_i)  
\label{eq:OFG2}
\end{equation}
\cite{dasgupta1986existence} show that a game with the same continuity features as Bertrand Competition with Capacity Constraints admits a Mixed Strategy Nash Equilibrium (MSNE). In particular, left lower semi-continuity of $\pi_i$ grants the existence of an MSNE. For an in-depth discussion of PSNE and MSNE, refer to \cite{alain2012games}. 

Proposition \ref{prop:diff-point} determines that the maximum transportation price for any supplier $i$ is $S_i+g_i$, and Proposition \ref{prop:diff-point2} establishes that the optimal value for transportation price strategy of suppliers in $G_1$ is $p_{G_2}^{\max}=\max_{i\in G_2} \{S_i + g_i\}$, and for suppliers in $G_2$ is $p_{G_3}^{\max}=\max_{i\in G_3} \{S_i + g_i\}$. 
Additionally, in a Bertrand Competition with Capacity Constraint where the pure-strategy sets are well-defined intervals, a mixed strategy will be given by a cumulative distribution function \citep{menache2011network}. In this case, this set is a well-defined interval $P_i = [0,S_i+g_i]$. Therefore, a mixed strategy for player $i$ is a cumulative distribution function $F_i : P_i \rightarrow [0, 1]$, where $F_i(p) = Pr\{p_i \leq p\}$. 
Considering these results and the existence of MSNE, Proposition \ref{prop:MSNE} follows.

\begin{proposition}\label{prop:MSNE}
    Let $N=G_1 \cup G_2$, i.e., the set of suppliers in $G_1$ and $G_2$, with $|N|=n$ members. Let $\pi_i$ and $\pi_i^\prime$ be the supplier $i$'s profit values when transporting all or part of its products with price $p_{i}$, respectively. Let $\bm{p}=(p_1,p_2,\ldots,p_n)$ be the price strategy selected by all suppliers. Let $o_i$ be the expected revenue of supplier $i$ given the competitors' strategy, and $\ell_{i}=\frac{o_i}{\pi_i + \pi_i^\prime}$. Then, the cumulative probability function of the mixed Nash equilibrium transportation price for supplier $i$ is:
    \begin{equation}
        F_i(\bm{p}) = \sqrt[n-1]{\frac{{\displaystyle \prod_{j\in N-\{i\}}\ell_j}}{\ell_{i}^{n-2}}} 
    \end{equation}
\end{proposition}
\begin{proof}
    To find the MSNE in an $n$-player game, one needs to find the mixed strategies of the other $n-1$ players to make the $n$th player indifference in its strategy. The strategy in this game is $\bm{p}$, and the expected revenue of supplier $i$ given strategy $\bm{p}$ is calculated as:
    \begin{equation}
       o_i = \pi_i \left (\displaystyle \prod_{j\in N-\{i\}} F_j(\bm{p}) \right) + \pi_i^\prime \left(\displaystyle \prod_{j\in N-\{i\}} F_j(\bm{p}) \right)        \label{eq:prop}
    \end{equation}
    Solving Equation \eqref{eq:prop} results in:
    \begin{equation}
          \prod_{j\in N-\{i\}}F_j\left(\bm{p}\right) =\frac{o_i}{\pi_i + \pi_i^\prime} =\ell_{i}
         \label{eq:prop2}
    \end{equation}
    which is the key to finding $F_i(\bm{p})$. From Equation \eqref{eq:prop2}, it can be shown that $F_i(\bm{p})\ell_i = F_j(\bm{p})\ell_j$. Next, replace $F_j(\bm{p})$ with $\frac{\ell_i F_i(\bm{p})}{\ell_j}$ for all $j$ in Equation \eqref{eq:prop2}. This substitution yields:
    \begin{equation}
        \ell_i = \frac{\ell_i^{n-1} F_i(\bm{p})^{n-1}}{\displaystyle \prod_{j\in N-\{i\}}\ell_j}\label{eq:prop3}
    \end{equation}
    Finally, separating $F_i(\bm{p})$ in Equation \eqref{eq:prop3} and cancelling out one $\ell_i$ from both sides results in:
    \[
    F_i\left(\bm{p}\right) = \sqrt[n-1]{\frac{{\displaystyle \prod_{j\in N-\{i\}}\ell_j}}{\ell_{i}^{n-2}}}
    \]
\end{proof}
\vspace{-8pt}Due to structural complexity, obtaining the expected revenue, and therefore, analyzing the mixed strategies in a direct way is not feasible. Furthermore, solving the set of equations addressed in Proposition \ref{prop:MSNE} requires a guess for at least one of the $o_i$ values, whose uniqueness goes beyond the scope of this paper. Therefore, we refrain from further discussion around the closed-form of the pricing equilibrium. However, as will be discussed in the upcoming sections, knowing the equilibrium prices is not required for obtaining the shipment quantities and addressing the research questions. 

\vspace{-5pt}
\section{Regulation}
\label{sec:Regulation}
The next step is to scrutinize the minimum quota regulation. Among the addressed cases in Section \ref{sec:shipping}, we focused on the cases where at least one of the suppliers is left with zero allocated transportation capacity, excluding case $C_1$. Determining a minimum quota requires taking into consideration both sides of the transaction. One side of the deal is the shipping company, and ensuring that the suggested transportation price is competitive and does not hurt the transportation company's  profitability. The other side of the transaction is called Maximum Revenue Entitlement (MRE) in Canada\footnote{In Canada, the transportation price of crude oil is determined by negotiation.}. MRE is the maximum revenue that the shipping company is allowed to gain by transporting one unit of a supplier's product. The MRE value depends on transportation cost and price, and shipment amount. Therefore, the minimum quota must provide all suppliers with transportation opportunity, given that those suppliers suggest a transportation price that is profitable for the shipping company while being capped by MRE.

Let $L_i < D_i$ be the transportation quota for producer $i\in I$. Obviously, the total assigned quota must be less than the total transportation capacity:
\begin{equation}
    \sum_{i} L_i \le T
\end{equation}
Furthermore, MRE sets a maximum transportation price for supplier $i\in I$:
\begin{equation}
   u_i(p_i-c_i) \le \mbox{MRE}_i \quad \Longrightarrow \quad  p_i \le \frac{\mbox{MRE}_i}{u_i}+c_i
\end{equation}
In summary, the necessary conditions for the minimum quota of supplier $i\in I$ are:
\begin{eqnarray}
    && p_i \le \frac{\mbox{MRE}_i}{u_i}+c_i\,, \nonumber \\
    && L_i \le D_i\,, \nonumber \\
    && \sum_{i} L_i \le T \nonumber 
\end{eqnarray}

While these conditions lay the foundation for defining the minimum quota, they do not explain the pricing policies of the suppliers. Note that competitiveness of $p_i$ is translated into different values for different groups of suppliers. For the suppliers with lower $S_i+g_i$ values, $L_i$ is the best obtainable arrangement. Hence, these suppliers will not offer transportation prices more than $c_i$. Suppliers seeking to secure transportation capacities more than the assigned minimum quota will be required to offer higher prices. The natural consequence of this observation is that the minimum quota regulation is required for suppliers of $G_3$ (and perhaps $G_2$) if the government deems these products are necessary for the society. Also, any supplier $i$ not in $G_3$ will be required to offer a price $p_i > c_i$ for all of their shipped units, even if they receive a minimum quota.

Let $\hat{T}=T-\sum_i L_i$ and $\hat{D}_i = D_i-L_i$. Adjusted versions of Algorithms \ref{alg:1} and \ref{alg:2} can be used to categorize and assign a transportation capacity to the suppliers, given that the total capacity and demands are revised to $\hat{T}$ and $\hat{D}_i$. 

\begin{remark}\label{rem:reg}
   Considering that the minimum quota does not affect marginal profitability of the transportation company and suppliers, and given that the reduction in the total transportation capacity in the adjusted model (with $\hat{T}$ and $\hat{D_i}$) is more than the total guaranteed quota for formerly $G_1$ and $G_2$ suppliers, i.e., $T-\hat{T} > \sum_{i\in G_1,G_2}{L_i}$, it can be shown that the $G_1$ category will not expand, and the former $G_2$ and $G_3$ suppliers will not promote and fall under $G_1$ and $G_2$ categories, respectively.
\end{remark}

The regulator sets a minimum quota to stop depriving some suppliers from transporting their products. The rationale behind this idea is the non-financial importance of moving these products to the market. One must quantify this importance in order to find the social optimum assignment of minimum quotas. 

Let $\beta_i$ be the financially scaled importance score of one unit of product $i$ reaching the market\footnote{For instance, if a type of grain is the targeted product, the financial importance score is the monetary burden of importing that type of grain from another country or region, if this supplier cannot secure a minimum transportation capacity.}, and subsequently, $\beta_i\times (\text{Transported Quantity})$ constitutes the social profit of transportation. All $u_i$ values are calculated by Algorithm \ref{alg:1} and the suppliers are categorized by the adjusted version of Algorithm \ref{alg:2}. Let $i_{G_2}$ be the supplier in $G_2$ when it is not empty. For any supplier in $G_1$, the optimal transported amount, $u_i$, is equal to $D_i$, for the supplier in $G_2$, 
$u_{i_{G_2}} = L_{i_{G_2}} + (T-\sum_i L_i)-\sum_{i\in G_1}(D_{i}-L_{i})$, and for any supplier in $G_3$, $u_i = L_i$. 
Hence, assuming  
$u_{i_{G_2}} > L_{i_{G_2}}$, the regulator solves the Social Utility Maximization Model (SUMM) as follows\footnote{If $u_{i_{G_2}} = L_{i_{G_2}}$, then the supplier in $G_2$ downgrades to $G_3$ to pay the lowest possible shipping price, i.e., $c_i$.}.
\begin{eqnarray}
\label{model:SUMM}
    \underset{p_{i},L_{i},\, i\in\{\cup_{k\in \{1,2,3\}} G_k\}}{\maximize}\,\Pi_S \; = &&
    \underbrace{\sum_{i\in G_1}D_{i}(S_{i}-p_{i})}_{G_1 \text{suppliers' profit}} + \underbrace{\sum_{i \in G_1} D_{i}\,\min\{(p_{i}-c_{i}),\frac{\mbox{MRE}_i}{u_i}\}}_{\text{shipping company's profit of } G_1} + \underbrace{\sum_{i\in G_1} \beta_{i}D_{i}}_{\text{social utility of } G_1}  \\[4ex]
   && \resizebox{0.7\textwidth}{!}{$+\underbrace{\,\,\sum_{i\in G_2}(Q+L_i)(S_{i}-p_{i})-\sum_{i\in G_2} g_{i}\left(D_{i}-(Q+L_i)\right)}_{G_2 \text{ suppliers' profit}} + \underbrace{\sum_{i\in G_2}(Q+L_i)\min\{(p_{i}-c_{i}),\frac{\mbox{MRE}_i}{u_i}\}}_{\text{shipping company's profit of } G_2} + \underbrace{ \sum_{i\in G_2} (Q+L_i) \beta_{i}}_{\text{social utility of } G_2} \nonumber $} \\[4ex]
    && \resizebox{0.7\textwidth}{!}{$+ \underbrace{\sum_{i\in G_3} L_{i}(S_{i}-c_{i})-g_{i}(D_{i}-L_{i})}_{G_3 \text{ suppliers' profit}} + \underbrace{0}_{\text{shipping company's profit of } G_3} + \underbrace{\sum_{i\in G_3} \beta_{i}L_{i}}_{\text{social utility of } G_3}  \nonumber $}
\end{eqnarray}
s.t. \vspace{-15pt}
\begin{eqnarray}
    && L_i \le D_i \,, \quad\;\;\;\qquad\qquad \forall i \nonumber \\
    && \sum_{i} L_i \le T \,, \nonumber \\
    && Q = (T-\sum_i L_i)-\sum_{i\in G_1}(D_{i}-L_{i})\, \nonumber
\end{eqnarray}
Note that here, $u_i$ values are treated as parameters and $Q$ is the shipment quantity of the sole supplier in $G_2$ in addition to its minimum quota.  
Also, since $G_1$ products are fully shipped, their assigned minimum quota will not appear in the objective function.

\begin{proposition}\label{prop:minquota}
    Any supplier $i$ in group $G_3$ who satisfies the following constraint will not be assigned with a minimum quota:
    \begin{equation}
        S_i+g_i+\beta_i-c_i < \sum_{j\in G_2}(S_{j}+g_{j}+\beta_{j}-c_{j})
    \end{equation}
\end{proposition}
\begin{proof}
    It can be shown that:
\[
\Pi_S = \sum_{i\in G_3}L_i\left(\left(S_i+g_i+\beta_i-c_i\right)-\sum_{j\in G_2}\left(S_{j}+g_{j}+\beta_{j}-c_{j}\right)\right)+K\,,
\]
    in which $K$ is independent from $L_i$. Then, maximizing this function requires setting $L_i=0$ for any supplier that satisfies $S_i+g_i+\beta_i-c_i < \sum_{j\in G_2}(S_{j}+g_{j}+\beta_{j}-c_{j}).$
\end{proof}

Proposition \ref{prop:minquota} shows that even when the regulator intervenes for the benefit of the whole society, setting a minimum quota for some producers is still not justified.

When analyzing model (\ref{model:SUMM}), if MRE is restricting, the maximum transportation price depends on the allocated capacity $u_i$, which in turn according to Problem (\ref{model:shipping}) depends on the equilibrium transportation prices, i.e., $\bm{p}$ in Proposition \ref{prop:MSNE}. 
If MRE is not restricting, the objective function of SUMM can be re-arranged as 
\begin{eqnarray}
    \Pi_S = && \underbrace{\sum_{i\in G_3}\left(L_i(S_i+g_i+\beta_i-c_i)-g_iD_i\right)}_{\text{$G_3$ suppliers}}+\underbrace{\sum_{i\in G_1}D_i(S_i+\beta_i-c_i)}_{\text{$G_1$ suppliers}} \nonumber \\
    && + \underbrace{\sum_{i\in G_2}\left((Q+L_i)(S_{i}+g_{i}+\beta_{i}-c_{i})-g_{i}D_{i}\right)}_{\text{$G_2$ supplier}}\,,\label{eq:final_obj}
\end{eqnarray}
in which the equilibrium competitive prices are not present. Here, $Q = (T-\sum_i L_i)-\sum_{i\in G_1}(D_{i}-L_{i})$.

\begin{corollary} 
By Equation \eqref{eq:final_obj}, the best set of minimum quotas are independent from the suggested competitive prices of the suppliers. 
\end{corollary}
This result allows the government to find an optimal solution for SUMM (model \eqref{model:SUMM}) without discussing the equilibrium transportation prices. 
Before proceeding with the solution algorithm, Corollary \ref{cor:that} is presented to lay the foundation for categorizing the suppliers.
\begin{corollary}\label{cor:that}
    Whether or not the regulator sets a minimum quota for $G_1$ and $G_2$ suppliers, their categories will remain unchanged.
\end{corollary}
This is important in the sense that $\Pi_S$ depends on the suppliers' categories, which in turn depends on $L_i$ values. Corollary \ref{cor:that} grants to start solving the problem by initially assuming that all suppliers belong to $G_3$. In other words, all suppliers will be assigned an initial minimum quota value.

\begin{algorithm}[!htb]
\protect\caption{\label{alg:3} Algorithm ${\sf SolveSUMM}$ takes the input parameters of the problem and finds a solution to model SUMM.}
\SetAlgoLined
\BlankLine
\KwIn{The set of suppliers $I$, production amount $D_i$, market price $S_i$, inventory cost $g_i$, transportation cost $c_i$, and financially scaled importance score $\beta_i$ for each product $i\in I$, and total transportation capacity $T$.}
\KwOut{Group category for each supplier $i\in I$.}
\BlankLine
\tcc{**********************************************************************************}
$\textrm{Max} \leftarrow 0$\;
$j \leftarrow 0$\;
$k \leftarrow 1$\;
\For{$L_i\in \{0,1,...,D_i\},\forall i\in I$ {\textrm{where}} $\sum_i L_i \le T$}{
    $\hat{T} \leftarrow T-\sum_iL_i$\;  
$G_1,G_2,G_3 \leftarrow \{\}$\;
$m \leftarrow |I|$\;
$J \leftarrow I$\;
$p_i \leftarrow S_i+g_i\,, \; i\in I$\;
\While {$T>0$ \AND $m>0$}{
    Set $\ell = \underset{i \in J}{\argmax}\{S_i+g_i-c_i\}$. If a tie occurs, break it randomly\;
    \eIf {$D_{\ell} \leq T$}{
        $\hat{T} \leftarrow \hat{T}-(D_{\ell}-L_{\ell})$\; 
        $G_1 \leftarrow G_1 \cup \{\ell\}$\;
    }{
        $T \leftarrow 0$\; 
        $G_2 \leftarrow G_2 \cup \{\ell\}$\;
    }
    $m \leftarrow m-1$\;
    $J \leftarrow I\backslash(G_1 \cup G_2)$\;
}
$G_3 \leftarrow J$\; 

\For{$i\in G_1$}{
    $u_i \leftarrow D_i$\;
}
\For{$i\in G_2$}{
    $u_i \leftarrow \hat{T}-\sum_{i\in G_1}(D_i-L_i)$\;
}
Calculate $\Pi_S$ from Equation \eqref{eq:final_obj}\;
\If{$\Pi_S > \mbox{Max}$}{
    $\textrm{Max} \leftarrow \Pi_S$\;
    $j \leftarrow k$\;
}
$k \leftarrow k+1$\;
}
$\Pi_S^* \leftarrow \textrm{Max}$\;
$L^* \leftarrow L_j$\;
\KwRet{$\Pi_S^*,L^*$}\;
\end{algorithm} 

This is important in the sense that $\Pi_S$ depends on the suppliers' categories, which in turn depends on $L_i$ values. Corollary \ref{cor:that} grants to start solving the problem by initially assuming that all suppliers belong to $G_3$. In other words, all suppliers will be assigned an initial minimum quota value. Algorithm \ref{alg:3} proposes a full-search for the solution of SUMM and disregards the equilibrium values and MRE. The output of Algorithm \ref{alg:3} is the set of minimum quotas for all suppliers who require such a quota.

In this section, the suppliers, the shipping firm and the government problems were presented. Also, the optimal and equilibrium solutions were investigated. In the next section, we use data from the Canadian market to showcase how the methods described in this paper can be used to solve real-world problems in Canada, which was the motivation of this research.

\begin{remark}
It should be noted that models defined by models \eqref{model:shipping} and \eqref{model:supplier} can have multiple solutions. The possibilities are:
 when in Algorithm \ref{alg:1}, $p_i-c_i=p_j-c_j$ for some $i, j \in I$; or,
    when in Algorithm \ref{alg:2}, $\underset{i \in J}{\argmax}\{S_i+g_i-c_i\} = \underset{j \in J}{\argmax}\{S_j+g_j-c_j\}$ for some $i, j \in J$. In this case, Algorithms \ref{alg:1} and \ref{alg:2} can find all the solutions by breaking the ties differently. However, it should be noted that finding a single solution is enough to proceed with the proposed algorithms, i.e., finding all the possible solutions is not required. In other words, the transportation company does not need to consider all the possible strategies, so long as the chosen strategy is optimal. 
\end{remark}

\section{Numerical Analysis: A Canadian Case Study}\label{Numerical}

The motivation of this research is the Canadian rail transportation system which had been the topic of controversy for the reasons addressed in the Introduction section. Here, we implement the introduced algorithms of Sections \ref{Model} and \ref{sec:Regulation}. Canadian rail transportation companies, CP and CN, ship several types of products, including but not limited to, petroleum products, grains, oil, wood and forest products, metal scraps. To discuss the applicability of the proposed models, we use real data for Canadian oil production and three types of grain: corn, barley and oat. In Canada, both crude oil and grains are strategically important for internal use and export and for the size of the industries developed around them.

The upside of this example is having a representative case of the Canadian transportation market while avoiding the data size complexity of the actual system. The downside of this example, however, is that the transportation capacity cannot be materialized if all of the products are not considered. To deal with this issue, the transportation capacity parameter is introduced as a sensitivity analysis variable. Hence, the solutions for different possible capacity values are obtained; furthermore, the intuition can be used by transportation companies to evaluate the benefit of capacity expansion. 

An important point to remember is that the measuring units in this study are bushel, year and Canadian Dollar (CAD). The considered transportation capacity is that of a whole year. Although the yearly process is not representative of the actual monthly production and shipping process, if we assume the uniform production and shipping every year, then yearly data reflects an acceptable approximation of the actual situation. The model parameters for corn, barley, oat and crude oil are presented in Table \ref{tab:data}. The details and sources of data are presented in Appendix \ref{app:data}. 

\begin{table}[H]
    \centering
    \caption{Model parameters for crude oil, corn, barley, and oat.}
    \begin{tabular}{|c|c|c|c|c|} \hline
      Product   & $D_i$ (bushels)      & $S_i$ (CAD) & $c_i$ (CAD) & $g_i$ (CAD)  \\ \hline
      Crude Oil & 128,185,505 & 15.3  & 1.297 & 0.79   \\
      Corn      & 440,916,666 & 0.09  & 0.06  & 0.0042 \\
      Barley    & 440,924,524 & 0.111 & 0.113 & 0.0049 \\
      Oat       & 275,577,827 & 0.11  & 0.118 & 0.0073 \\ \hline
    \end{tabular}
    \label{tab:data}
\end{table}

An important but missing piece of information is the total transportation capacity of the rail transportation company; the capacity depends on the shipping company's number of cars, the company's speciality, and the railing system's maintenance. Analyzing these data creates complications beyond the purpose of the numerical analysis for the presented models. Also, working with the real capacity relates only if all of the products are looked at in the analysis, not a subset of the shipped products. Therefore, we perform sensitivity analysis on total capacity and study its impact on the objective function. 

We study two scenarios under the assumption that MRE is not restrictive. First, we investigate a scenario where all products are equally important for the society. Next, we concentrate on a specific scenario where the left out product is more important than at least one of the transported products. This classification allows for a more comprehensive scrutiny of the possible market scenarios.

\subsection{Scenario 1: Equally Important Products}
Since all products are equally important, $\beta_i=\beta=0.1$ for all products.

\textbf{Without regulation:} in the absence of regulation, the shipping company selects the suppliers who propose higher prices relative to their transportation cost. Assuming the transportation capacity of 650,000,000 bushels per year, applying Algorithm \ref{alg:2} results in Table \ref{tab:no_reg}.

\begin{table}[H]
    \centering
    \caption{Capacity allocation under no regulation scenario.}\label{tab:no_reg}
    \begin{tabular}{|c|c|c|c|} \hline
      Products  & $S_i+g_i-c_i$ & Group & Transported Amount \\ \hline
      Crude Oil & 14.793  & $G_1$ & 128,185,505  \\ 
      Corn      & 0.0342  & $G_1$ & 440,916,666  \\
      Barley    & 0.0029  & $G_2$ & 80,897,829  \\
      Oat       & -0.0153 & $G_3$ & 0  \\ \hline 
    \end{tabular}
\end{table}
\vspace{-8pt}
Note that none of the suppliers gains more profit by reducing transportation amount. Table \ref{tab:no_reg} demonstrates that under the no regulation scenario, only crude oil and corn suppliers achieve the transportation capacity they need; barley suppliers ship less than 20\% of the total production, and oat producers will not be allocated any transportation capacity. Under this scenario ($L_{i}=0\,,\; \forall i\in G_3$), the objective function of the regulator, according to Equation \eqref{eq:final_obj}, is 1,869,321,483. 

\textbf{With regulation:} when there is a minimum quota regulation, the regulator needs to set $L_i$'s such that Equation \eqref{eq:final_obj} is maximized. The first step is determining if the only supplier (product) in category $G_3$ deserves a minimum quota. According to Proposition \ref{prop:minquota}, oat receives a minimum quota only if its total of marginal value to the transporter and society is larger than that of barley, which is not the case in the current scenario; this is because the social value of all products are assumed to be equal.

Since oat will not receive a minimum quota, applying Algorithm \ref{alg:3} will yield the same value for the objective function specified by Equation \eqref{eq:final_obj} as without regulation scenario. Intuitively, it can be concluded that if all products carry equal value for the society, minimum quota regulation does not help the suppliers with left-out products.

\subsection{Scenario 2: Unequally Important Products}
Let $\beta_{Barley}=\beta_{Oil}=\beta_{Corn}=0.1$ and, in a virtual reality, $\beta_{Oat}=1$. The only difference in the no regulation case is the regulator's objective function, which makes no difference in the absence of regulation. In other words, the ``without regulation'' scenario does not require further investigation.

\textbf{With regulation:} The suppliers' categories will not change from the shipping company's point of view. Therefore, the values presented in Table \ref{tab:no_reg} remain unchanged. However, the objective function value of Equation \eqref{eq:final_obj} will be altered; the overall value of oat may be greater than barley, and therefore, it may deserve a minimum quota. Based on Proposition \ref{prop:minquota}, Oat receives a minimum quota if $S_{Oat}+g_{Oat}+\beta_{Oat}-c_{Oat} > S_{Barley}+g_{Barley}+\beta_{Barley}-c_{Barley}$, which is true. Since it is declared that $L^*\neq 0$, Algorithm \ref{alg:3} will be applied to find the optimum solution of SUMM. The result of Algorithm \ref{alg:3} follows in Table \ref{tab:RegUneqalImportance}.
\begin{table}[H]
    \centering
    \caption{Capacity allocation under scenario 2 and with regulation.}\label{tab:no reg}
    \begin{tabular}{|c|c|c|c|} \hline
      Products  & Original Grouping & New Grouping & Transported Amount \\ \hline
      Crude Oil & $G_1$  & $G_1$ & 128,185,505  \\ 
      Corn      & $G_1$  & $G_2$ & 246,236,668  \\
      Barley    & $G_2$  & $G_3$ & 0  \\
      Oat       & $G_3$ & $G_3$ & 275,577,827  \\
       \hline 
    \end{tabular}
    \label{tab:RegUneqalImportance}
\end{table}
\vspace{-8pt}
 We have $u_{\text{Barley}} = 0$ because $u_{\text{Oat}} = D_{\text{Oat}} = L_{\text{Oat}}$; note that there is no upper bound on minimum quota by definition (although it may not be practical). As one can observe from the results, the social value of a product may be so high that a less profitable product, from the shipping company's point of view, replaces a more profitable product; barley in our case. In other words, the minimum quota regulations help the left-out suppliers if the relative social value of their product is so high that its overall social and financial profit surpass, at least, one of the products which is being fully or partially shipped. Putting differently, the profit of the supplier by itself does not guarantee a minimum quota in the social utility function.

Assume that the regulator defines an artificial cap for the minimum quota, which is a plausible scenario for numerous reasons such as lobbying, avoiding favoritism, regional political pressure, etc. If the hypothetical minimum quota is no more than 20\% of the production, then the revised transportation quantities are according to Table \ref{tab:RegUneqalImportanceQuotaRestriction}.

The method to calculate the revised quantities is as follows. First, keep the same overall rank of the products, i.e., crude oil, oat, corn and barley. Then, allocate the shipping quantity with a small nuance that if a product is in the list only after considering the social value, its allocated shipping capacity should not be more than the stipulated percentage. Finally, continue the allocation until the entire capacity is assigned.

\begin{table}[t]
    \centering
    \caption{Capacity allocation under Scenario 2, with regulation and minimum quota restriction.}\label{tab:no reg}
    \begin{tabular}{|c|c|c|c|} \hline
      Products  & Original Grouping & New Grouping & Transported Amount \\ \hline
      Crude Oil & $G_1$  & $G_1$ & 128,185,505  \\ 
      Corn      & $G_1$  & $G_1$ & 440,916,666  \\
      Barley    & $G_2$  & $G_3$ & 25,782,263  \\
      Oat       & $G_3$ & $G_3$ & 55,115,565  \\ 
       \hline 
    \end{tabular}
    \label{tab:RegUneqalImportanceQuotaRestriction}
    \vspace{-8pt}
\end{table}
As one can see, considering minimum quota shares the opportunity to less profitable products, or even deprives some less socially, yet more financially profitable products from shipping. Therefore, this idea shall be implemented with in-depth understanding and pre-calculation to ensure the fairness of all measurable and immeasurable aspects of production and transportation.

Figure \ref{fig:SummaryofNumerical} summarizes the results of the above scenarios and illustrates the shipped quantities (tableau \ref{fig:commodityQuantity}) and their share from the total transportation capacity (tableau \ref{fig:commodityPercentage}) under each scenario. It can be noticed that crude oil producers are able to transport their entire production under all three scenarios. However, introducing minimum quota regulations pushes less competitive products such as barley and oat to a different group, and hence, drastically change their transported amount. In other words, the regulator must remain vigilant in determining the minimum quotas to avoid unwanted consequences, which further signifies the importance and impact of the present study.

\begin{figure}[b!]
    \centering
    \begin{subfigure}[b]{0.45\linewidth}
        \includegraphics[width=\linewidth]{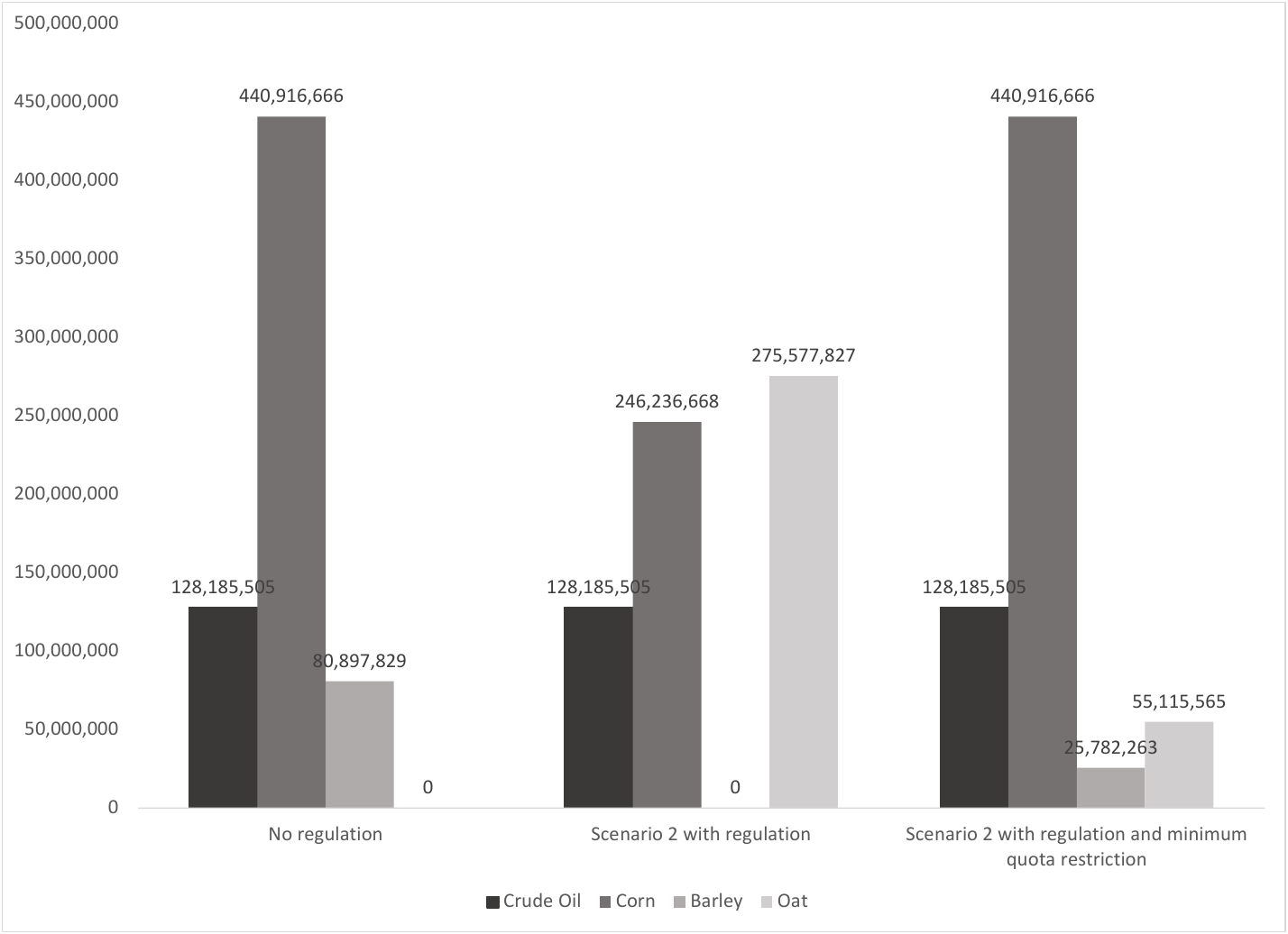}
        \caption{}
        \label{fig:commodityQuantity}
    \end{subfigure}
    \qquad
    \begin{subfigure}[b]{0.45\linewidth}
        \includegraphics[width=\linewidth]{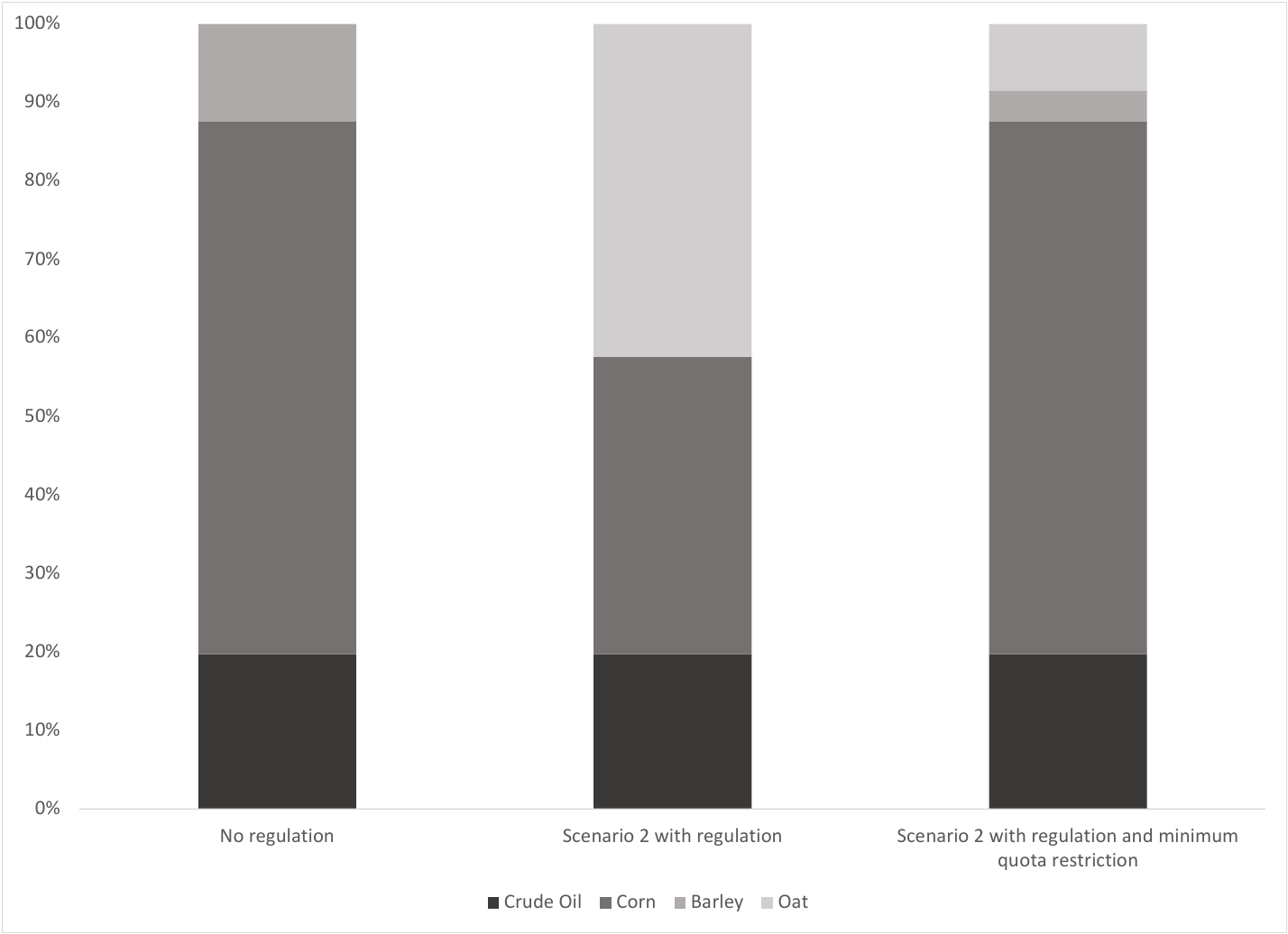}
        \caption{}
        \label{fig:commodityPercentage}
    \end{subfigure}
    \caption{The allocation of transportation capacity for different commodities under three different scenarios. The shipment quantity of each commodity is shown in (\ref{fig:commodityQuantity}) and the percentage of the total shipping capacity allocated to each commodity is presented in (\ref{fig:commodityPercentage}).}
    \label{fig:SummaryofNumerical}
\end{figure}

\subsection{Transportation Capacity} This section examines the impact of transportation capacity on the regulator's objective function value. The purpose of this section is scrutinizing the results when: 1) there is a lack of available data about the total transportation capacity; and, 2) one needs to avoid analyzing the data about all products that a shipping company transports during a time period. 

According to Proposition \ref{prop:minquota}, and assuming that all products are equally important, none of the products in $G_3$ receive a minimum quota. Also, note that $T$ is fundamental in determining the categories. After these remarks, we proceed by investigating the impact of total transportation capacity on the regulator's objective function, when the products carry equal importance.

\begin{table}[H]
    \centering
    \caption{Impact of transportation capacity on the allocated capacity}
    \label{tab:capacity-case 1}
    \resizebox{\columnwidth}{!}{
    \begin{tabular}{|c|c|c|c|c|c|}
      \hline
       Area & Transportation Capacity             & Corn  & Barley & Oat   & Crude Oil \\ \hline
       $A_1$ & $\bm{T} < D_{Oil}$                 & $G_3$ & $G_3$  & $G_3$ & $G_2$     \\
       $A_2$ & $D_{Oil} \le \bm{T} < D_{Oil}+D_{Corn}$ & $G_2$ & $G_3$  & $G_3$ & $G_1$ \\
       $A_3$ & $D_{Oil}+D_{Corn} \le \bm{T} < D_{Oil}+D_{Corn}+D_{Barley}$ & $G_1$ & $G_2$ & $G_3$ & $G_1$ \\
       $A_4$ & $D_{Oil}+D_{Corn}+D_{Barley} \le \bm{T} < D_{Oil}+D_{Corn}+D_{Barley}+D_{Oat}$ & $G_1$ & $G_1$ & $G_2$ & $G_1$ \\
       $A_5$ & $D_{Oil}+D_{Corn}+D_{Barley}+D_{Oat} \le \bm{T} $ & $G_1$ & $G_1$ & $G_1$ & $G_1$ \\ \hline
    \end{tabular}
    }
    \vspace{-8pt}
\end{table}
Considering Table \ref{tab:capacity-case 1}, the regulator's objective function relative to the transportation capacity is shown in Figure \ref{fig:capacity-case 1}. Note that $\Pi_S$ is concave and increasing in $T$. Also, the social value of each product is concave and increasing due to diminishing marginal social benefits.

\begin{figure}[H]
    \centering
    \begin{tikzpicture}
        \begin{groupplot}[group style={group name= my plots,group size=1 by 1, horizontal sep=1.5cm, vertical sep=1.2cm},height=7.5cm,width=15.5cm,ticklabel style={font=\small}]
		
		\nextgroupplot[ylabel={$\Pi_S\times 10^7$},xlabel={$T\times 10^7$}]
		
		\addplot[color=black, no marks, thick,solid, smooth] coordinates {
			(0,	-10.72906473)
            (1,	4.163935275)
            (2,	19.05693527)
            (12,   167.9869353)
            (12.8185505,	180.1776079) 
            (13,	180.2250317)
            (16,	180.6276317)
            (49,	185.0562317)
            (52,	185.4588317)
            (55, 	185.8614317)
            (56.910217,	186.1177829) 
            (57,	186.1039482)
            (99,	190.4257482)
            (101.0026696,	190.6318229) 
            (102,	190.9320296)
            (128.5604522,	193.5694825) 
            (128.5604523,	193.3683107)
            (130,	193.3683107)
            (150,	193.3683107)

		};
		\draw [dashed, very thick] (101.0026696,210) -- (101.0026696,-1);
		\draw [dashed, very thick] (56.910217,210) -- (56.910217,-1);
		\draw [dashed, very thick] (12.8185505,210) -- (12.8185505,-1);
		\draw [dashed, very thick] (128.5604522,210) -- (128.5604522,-1);
		\node at (0,150) {\small $A_1$};
		\node at (35,150) {\small $A_2$};
		\node at (80,150) {\small $A_3$};
		\node at (115,150) {\small $A_4$};
		\node at (140,150) {\small $A_5$};
		\end{groupplot}
    \end{tikzpicture}
    \caption{Impact of altering transportation capacity ($T$) on the regulator's objective function.}
    \label{fig:capacity-case 1}
\vspace{-5pt}
\end{figure}
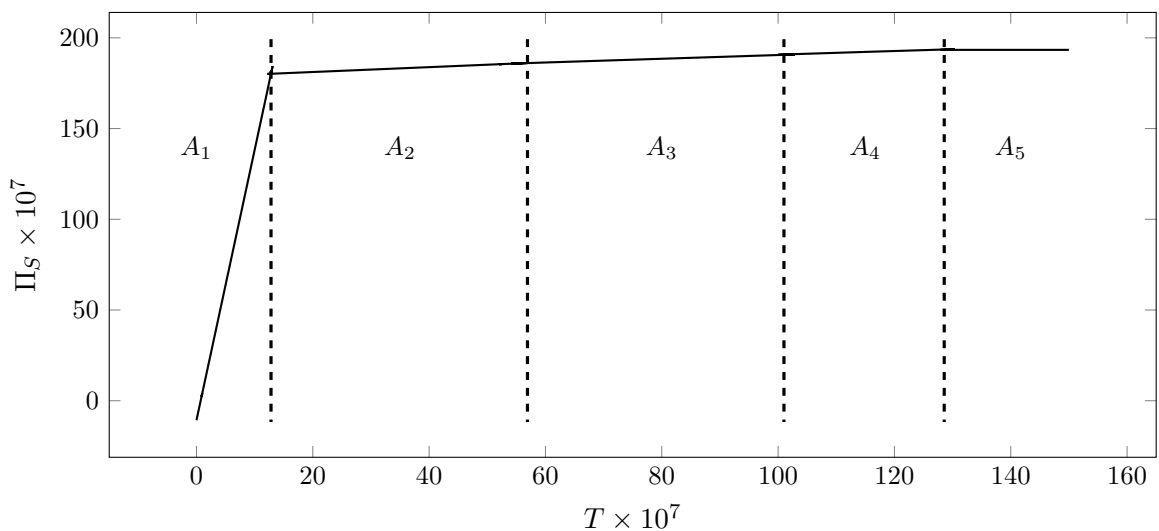

As Figure \ref{fig:capacity-case 1} depicted, the marginal value of adding transportation capacity is larger when the total capacity is low. In other words, marginal value of adding transportation capacity has an inverse relation with its size, with the slope being dependent on the total market production.

These marginal values can be used by the regulator to design an incentive system for the shipping companies based on periodic market production and total available capacity, where the transporter receives a financial reward for adding to the system's capacity. This could be the topic of future studies.

Finally, one can claim that for the case of unequally important products, the total social benefit derived from increasing the transportation capacity is the same as the equally important products; the only difference is that the positive slope of the curve is higher when products are not equally important because more profitable products are ignored for the sake of more socially valuable but less profitable products. \vspace{-5pt}

\section{Conclusion}
\label{Conclusion}
This paper aims to improve supply chain resiliency while considering the society's utility when a number of suppliers of critical commodities compete with each other for transportation capacity. This issue is a known problem in Canada and parts of the United States, where it is possible that crop and oil producers compete with each other when demand for transportation capacity is higher than usual, which has led to enactment of ``Canada's Fair Rail for Farmer's Act'' in Canada, as well as discussions for enactment of similar laws in the United States. Satisfactory implementation of such laws and government's intervention in the context of free markets requires multi-faceted analyses. This paper is the first effort to model the impact of the mentioned intervention mathematically.

Consequently, we discussed the equilibrium shipment quantities decided by a shipping company, equilibrium transportation prices offered by competing suppliers, and minimum quota regulation set by the government. In particular, this research was motivated by a Canadian rail transportation case; consequently, a framed case study accompanies the theoretical research.

On the theoretical side, a rail transportation company decides on the quantity of the products to ship based on the transportation prices offered by the suppliers. Given the limited transportation capacity, the suppliers (each having a distinguished product) compete \'a la Nash to have their products moved from the warehouse to the point of sale. Considering that some suppliers do not have large enough profit margins to offer competitive prices, the government sets a minimum transportation quota to make sure that all important products reach the market. This paper proposed mathematical models for the suppliers, the transporter, and the government. We characterized solution algorithms for the shipping company, the suppliers and the government to maximize the social utility. 

Finally, a Canadian case was studied to address the main research questions of this paper and shed light on the proposed solution methods. For the Canadian case study, four products were considered: three grains and crude oil. Two scenarios were investigated. One scenario assigns equal weights to all products; in the other scenario, the social value of the products are considered to be unequal.

It was outlined that if the social value of all products are equal, the product not shipped due to lower marginal revenue for the supplier (and consequently lower cap on the suggested transportation price), does not deserve a minimum quota to maximize social utility. In the Canadian case, this product was discovered to be oat. For the unequal weights scenario, social value of oat was set such that it surpassed the social value of barley and corn. This alteration made oat worthy of receiving a minimum transportation quota. In other words, in the absence of a cap on the minimum quota amount, the less profitable product (from the shipping company's point of view) can replace more profitable products. 

Furthermore, it was revealed that the marginal value of transportation capacity increases as the demand for transportation increases. However, after a certain point, the importance of increasing transportation capacity diminishes. It was reasoned that the regulator can employ similar analyses when encouraging the enhancement of transportation capacity. For instance, an incentive mechanism for increasing the capacity can be designed to handle seasonal or periodical market fluctuations.

Besides the factors considered in this paper, such as the market structure and the pricing schemes, that impact the shipper-supplier dynamics, product specifications such as price elasticity of demand and the location of transportation demand and their impact on those dynamics also deserve a thorough analysis and is an interesting direction for future research. Moreover, within the assumed market structure of this research, price was considered to be the sole factor driving the supply and demand dynamics. However, there are other important factors that can be studied in the future research efforts. The list includes history of business relationship between a supplier and a carrier, economies of scale, the expected future volume of transportation service requests by a supplier, and compatibility between transportation equipment and product specifications.
Considering some of these new factors could also add interesting dimensions to the problem. For example, while considering the location of transportation demand, one could also incorporate a geographical equity factor and use an equitable districting paradigm (see \cite{behrooziCarlsson2020Springer}) to allocate transportation capacity to different geographical regions in an equitable way.
Furthermore, one could also analyze the impact of each of the mentioned scenarios on the total society's utility and investigate whether the government should invest on expanding the transportation firm's capacity. 
These directions, plus expanding the proposed configurations to multi-period models, possibly with inventory, time-value-of-money, and perishability considerations, are also inspirational for future studies in this field.

Another stimulating path for future research is modelling the discussed problem as an auction, where the shipping company accepts bids from the suppliers on its available transportation capacity. This case differs from our model in the sense that the bidders will be able to obtain information about the other bids and change their offerings accordingly.

Finally, it is evident that the impact of natural disasters and climate change-related disturbances such as flooding, draughts or wildfires, and uncertainties related to geopolitical tensions will become more apparent in the near future, which have proven to cause significant disruptions in the commodity production as well as the available transportation capacity and options. Such disruptions are very difficult to forecast and have a profound impact on the supply and demand equilibrium. As such, generalizing the developed mathematical models for the case of stochastic or uncertain production and transportation capacity will make an interesting direction for research in this field and can lead to more robust and resilient supply chains.

\section*{Acknowledgement}
The authors are grateful for the support from Canada's NSERC Discovery Grant [\#2017-03743] which made this research possible.

{\small
\singlespacing

\bibliographystyle{apalike}
\bibliography{Transportation.bib}
}

\newpage 
\appendix
\part*{Appendix}

\section{Numerical Case Data}
\label{app:data}
\textbf{Corn:} nationally, corn's yield is 148.2 bushels per acre.\footnote{Statistics Canada; URL: https://www150.statcan.gc.ca/n1/daily-quotidien/190828/dq190828a-eng.htm. Retrieved May 2020.} The total cost of growing corn per acre in 2020 is \$532.56\footnote{URL: https://www.gov.mb.ca/agriculture/farm-management/production-economics/pubs/cop-crop-production.pdf. Retrieved on May 2020.}, resulting to the cost of \$3.59 per bushel of corn, which weights 56 lbs per bushel\footnote{URL: http://www.ilga.gov/commission/jcar/admincode/008/00800600ZZ9998bR.html. Retrieved on May 2020.}. Therefore, each pound of corn costs almost 6 cents ($c_i$). The storage cost, which is assumed to be the same for other types of grains, is 22 to 25 cents per bushel per year (average of 23.5). Hence, the storage cost ($g_i$) is calculated to be 0.42 cent per pound per year. The total corn production during 2019 harvest year was 13.6 million tonnes, with the average price of \$197.69 per tonne (based on data for the province of Ontario)\footnote{Grain Farmers of Ontario; URL: https://gfo.ca/marketing/average-commodity-prices/historical-barley-prices/. Retrieved on May 2020.}. Finally, the total per pound selling price of corn was approximately 9 cents ($S_i$), and the annual production was 440,916,666 bushels ($D_i$). 

\textbf{Barley:} in Canada, the average yield of barley is 66.4 bushels per acre\footnote{All sources are the same as corn.}. The total cost of growing barley per acre in 2020 is \$359.89, resulting to the cost of \$5.42 per bushel of barley, which weighs 48 lbs per bushel. Therefore, each pound of corn costs almost 11.3 cents ($c_i$). The storage cost is 23.5 cents per bushel per year, which translates to 0.49 cent per pound per year ($g_i$). The total barley production during 2019 harvest year was 9.6 million tonnes, with the average price of \$235.45 per tonne (based on Ontario data). Furthermore, the total per pound selling price of barley was 11.1 cents ($S_i$), and the yearly production was 440,924,524 bushels ($D_i$)\footnote{The governmental reports of Canada shows a negative return of investment for Barley for 2020. See https://www.gov.mb.ca/agriculture/farm-management/production-economics/pubs/cop-crop-production.pdf}.

\textbf{Oat:} nationally, the yield for oat is 89.6 bushels per acre\footnote{All sources are the same as of corn.}. The total cost of growing oat per acre in 2020 is \$338.17, resulting to the cost of \$3.77 per bushel of oat, which weights 32 lbs per bushel. Therefore, each pound of oat costs almost 11.8 cents ($c_i$). The storage cost is 23.5 cents per bushel per year, or 0.73 cent per pound per year ($g_i$). The total oat production during the 2019 harvest year was 4 million tonnes, with the average price of \$240.48 per tonne (based on Ontario data). The selling price for oat was 11 cents per pound ($S_i$), and the annual production was 275,577,827 bushels ($D_i$).

Note that while the total production amounts are reported for 2019, the recorded per acre costs belong to the province of Manitoba, Canada, during 2020. The reason is that the base values are approximated in January 2020, according to 2019 values and next year estimation.

\textbf{Crude oil:} we specifically select \textit{Mixed Sweet Blend Edmonton} as a representative of petroleum products, with the average price of 434 CAD\footnote{Government of Canada; URL: https://www.nrcan.gc.ca/our-natural-resources/energy-sources-distribution/clean-fossil-fuels/crude-oil/oil-pricing/18087. Retrieved on May 2020.} per cubic meter during 2019; this equals \$15.3 per bushel ($S_i$)\footnote{The measurement unit of volume for all products is translated to bushels for the sake of normalization and comparison.}. The projected production rate for Western Canada light crude oil in 2020 is 4,517,138 cubic meters, which equals 128,185,505 bushels ($D_i$). As of 2014, production cost, excluding capital spending and gross taxes, is \$5.85 per barrel or \$1.297 per bushel. For the transportation cost, unfortunately we did not have access to Canadian data. But since US total production price is close to Canadian production price, we rely on Shale cash cost per barrel which is \$3.59\footnote{URL: https://www.fool.com/investing/2017/03/19/you-wont-believe-what-saudi-arabias-oil-production.aspx, retrieved on June 2020}, or equivalently, \$0.79 per bushel.

\end{document}